\documentclass[11pt,reqno]{amsart}
\usepackage[utf8]{inputenc}
\usepackage{amssymb}
\usepackage{amsmath}
\usepackage{amsfonts}
\usepackage{amsthm,latexsym,color}
\pdfoutput=1

\usepackage{pifont}
\usepackage{verbatim}
\usepackage{amsfonts}

\font\TenEns=msbm10 \font\SevenEns=msbm7 \font\FiveEns=msbm5
\newfam\Ensfam

\textfont\Ensfam=\TenEns \scriptfont\Ensfam=\SevenEns
\scriptscriptfont\Ensfam=\FiveEns

\def\R {\mathbb R}
\def\C {\mathbb C}

\newtheorem{theorem}{Theorem}
\theoremstyle{plain}

\newtheorem{corollary}{Corollary}[section]

\newtheorem{definition}{Definition}

\newtheorem{lemma}{Lemma}[section]

\newtheorem{proposition}{Proposition}[section]
\newtheorem{remark}{Remark}[section]

\numberwithin{equation}{section}
\numberwithin{equation}{section}
\title[Inhomogeneous Nonlinear Schr\"odinger Equation]{Global existence and Scattering for the Inhomogeneous Nonlinear Schr\"odinger Equation}
\author[L. Aloui]{Lassaad Aloui}
\address{Lassaad Aloui, Universit\'e de Tunis El Manar, Facult\'e des Sciences de Tunis, D\'epartement de
Math\'ematiques, 2092 Tunis,
Tunisie.}
\address{and University of Sousse,  Laboratoire LAMMDA, Sousse, Tunisie.}
\email{lassaad.aloui@fsg.rnu.tn}
\author[S. Tayachi]{Slim Tayachi}
\address{Slim Tayachi,  Universit\'e de Tunis El Manar, Facult\'e des Sciences de Tunis, D\'epartement de
Math\'ematiques, Laboratoire  \'Equations aux D\'eriv\'ees
Partielles LR03ES04,  2092 Tunis,
Tunisie.}
\email{slim.tayachi@fst.rnu.tn}
\date{\today }
\begin{document}
\begin{abstract}
In this paper we consider  the inhomogeneous  nonlinear Schr\"odinger  equation $i\partial_t u +\Delta u  =K(x)|u|^\alpha u,\; u(0)=u_0\in H^s(\R^N),\; s=0,\,1, \; N\geq 1,\; |K(x)|+|x|^s|\nabla^s K(x)|\lesssim |x|^{-b} ,\; 0<b< \min (2, N-2s),\; 0<\alpha<{(4-2b)/(N-2s)}$.  We obtain novel results of global existence for oscillating initial data and scattering theory in a weighted $L^2$-space  for a new range $\alpha_0(b)<\alpha<(4-2b)/N$. The value $\alpha_0(b)$ is the positive root of $N\alpha^2+(N-2+2b)\alpha-4+2b=0,$ which extends the  Strauss exponent known for $b=0$. Our results improve the known ones for $K(x)=\mu|x|^{-b}$, $\mu\in \mathbb{C}$ and apply for more general potentials. In particular, we show the impact of the behavior of the potential at the origin and infinity on the allowed range of $\alpha$. Some decay estimates are also established for the defocusing case. To prove the scattering results, we give a new criterion taking into account the potential $K$. \end{abstract}

\maketitle
\section{Introduction}
In this paper, which is a continuation of our previous article \cite{AT-1},  we investigate the global existence and the asymptotic behavior for the inhomogeneous nonlinear Schr\"odinger equation

\begin{equation}\label{INLS}
  i\partial_t u +\Delta u  =K(x)|u|^\alpha u,
     \end{equation}
     with initial data
     \begin{equation}  \label{InitialINLS}
                 u(0,.)=u_0\in H^s(\R^N).
\end{equation}
Here $u=u(t,x)\in\C,\; t\in  \R,\; x\in \R^N$, $N\geq 1$, $s= 0$ or $s=1$ and $\alpha>0$.
The potential $K$ is a complex valued function satisfying some hypothesis. In particular, $$K(x)=\mu |x|^{-b}$$ and $$K(x)=\mu (1+|x|^{2})^{-\frac{b}{2}},$$
$b>0$, $\mu\in \mathbb{C}$, will be considered.

Equation  \eqref{INLS} with a constant function $K$, corresponds to the  standard nonlinear Schr\"odinger equation.  The case where $K$ is non constant and bounded is considered in  \cite{M,RS}. The unbounded  potential case  is also treated in \cite{Ch,ChG,Z}, where $K(x)=|x|^b$. Here we consider     \eqref{INLS} with potential having a decay like $|x|^{-b}$  at infinity and may be singular at the origin. This kind of equations appears in diverse branches of physics such as  nonlinear optics. See for example \cite[Section 6]{Genoud}. The local theory for \eqref{INLS} has been established in \cite{AT-1,Dinh, FarahGuzman, GS, Guzman, KimLeeSeo, LeeSeo} under the conditions $(K_1)-(K_2)$ below. We mention also that the study of the standing waves for \eqref{INLS} is done in \cite{DeBF,FO,LWW}. Here, we are mainly interested in global existence and small energy scattering for oscillating data. Also we  study the  decay estimates and the complete scattering  for the defocusing case.

For the standard nonlinear Schr\"odinger equation, there are vast amount results on the asymptotic behavior. See, among many, \cite{CW3, Cazenave,GinibreVelo,Tsutsumi} and the references therein. If $K,\nabla K\in L^\infty$, a results of global existence and scattering have been obtained in \cite{CW3} for oscillating initial data. For the case $K(x)=\mu |x|^{-b}$ with $b>0$, $\mu\in \mathbb{R}$ and $\frac{4-2b}{N}<\alpha<\frac{4-2b}{N-2}$ similar results are established in \cite{Guzman}. Our aim here is to improve this result in terms of the allowed values of $\alpha$ and on the smallness of initial data similar to the case $b=0$ in \cite{CW3}. To do this we exploit the scaling of the equation and the blow up criterion given in \cite{AT-1} and we establish  the Strichartz  estimates for non-admissible pairs to handle the potential $K$.

Our method allows us to consider potentials having different behavior near the origin and at infinity. In particular, we show that  more the potential decreases wider is the range of allowed  $\alpha$ giving scattering. For instance, if $K$ is regular and its decay at infinity exceeds $|x|^{-2}$, then the scattering for oscillating data occurs for all $0<\alpha<4/(N-2)$.

We are also interested in the scattering without smallness conditions on initial data, then we consider the defocusing case that is $K\geq 0$. We note that in this setting some scattering results are obtained in the weighted $L^2$ space for $\frac{4-2b}{N}<\alpha<\frac{4-2b}{N-2}$. See  \cite{Dinh2}. The proof of \cite{Dinh2} is based on some decay estimates. It seems that this type of estimates do not allow values of $\alpha$ less than $ (4-2b)/N$. Our second aim is to go down this last value of $\alpha$. To do this we establish a scattering criterion which is expressed in terms of rapidly decay of the solution in weighted Lorentz space taking into account the potential $K$. We also refined the decay estimate of \cite{Dinh2}.

Our scattering results concern the weighted $L^2$ space and they are valid for $\alpha>\alpha_0(b)$,  where  $\alpha_0(b)$ is defined by \eqref{alpha0} below which coincides with the known one  for $b=0$. As well as, our proof unifies the cases $b=0$ and $b>0$.  We mention that recently, an $L^2-$ scattering result is obtained in \cite{AIMMU}.

At the end of this paper, we are interested in the energy scattering for \eqref{INLS} with $\alpha$ in the inter-critical range. In the focusing setting, some results for this type of problem are  established in \cite{DK,Cam,Cam-Car,CFGM,FarahGuzman,FarahGuzman2} for $K(x)=|x|^{-b}$ and in \cite{CL,CHL,MMZ} and references therein for more general potentials. Here we give a scattering criterion in $H^1$ regardless the sign of $K$ (see Proposition \ref{refineddecayandscatteringH1} below) which allows us to give an alternative proof of some known results in \cite{Dinh2}.

In order to state our results, we need the following.
\begin{definition}[Admissible pair]
\label{defpairadmisible}
We say that  $(r,p)$ is an admissible pair if it satisfies
\begin{eqnarray}
\label{defAdmiss2}
\frac{2}{r}+\frac{N}{p}= \frac{N}{2}
\end{eqnarray}
and
$$2\leq p\leq \frac{2N}{N-2} \;\; (2\leq p<\infty \; \text{ if } N=2,\;  2\leq p\leq \infty  \text{ if } N=1).$$
\end{definition}
\noindent In this paper only the admissible pairs $(r,p)$ with $p<\infty$ will be considered.

We study  the problem \eqref{INLS}-\eqref{InitialINLS} in its integral version
\begin{equation}
\label{intINLS}
u(t)=e^{it\Delta}u_0-i\int_0^te^{i(t-s)\Delta}(K|u|^\alpha u(s))ds,
\end{equation}
where $e^{it\Delta}$ is the free Schr\"odinger group.
We consider the following conditions on $K$.
$$ (K_1)\quad |K(x)|\lesssim |x|^{-b},$$
$$(K_2)\quad|\nabla K(x)|\lesssim |x|^{-b-1},$$
 for $ x\in \R^{N}\setminus \{0\}$, where $0<b<\min (2,N)$.

We now suppose that
\begin{equation}
\label{conditionsNBalpha}
  0<\alpha<{4-2b\over (N-2)_+},
\end{equation}
where $r_+=\max (r,0)$, for a real number $r$.
Let $\alpha_0(b)$ be the positive root of the equation
\begin{equation}
\label{equationalpha0}
N\alpha^2+(N-2+2b)\alpha-4+2b=0.
\end{equation}
That is
\begin{equation}
\label{alpha0}
\alpha_0(b)={-(N-2+2b)+\sqrt{(N-2+2b)^2+4N(4-2b)}\over 2N}.
\end{equation}
Since $b<2$, then $\alpha_0(b)$ is well defined and verifies
    $${4-2b\over N+2}<\alpha_0(b)<{4-2b\over N}.$$
 See \cite{Tsutsumi, St}  for the case $b=0$. By natural extension we have $\alpha_0(2)=0.$

For $\alpha$ satisfying \eqref{conditionsNBalpha}, we introduce the following positive real numbers
\begin{equation}
\label{aP}
\varrho={N(\alpha+2)\over N-b},\; a={2\alpha(\alpha+2)\over 4-2b-\alpha(N-2)}.
\end{equation}
Clearly, $2<\alpha+2<\varrho<{2N\over (N-2)_+}$ and $1 < a<\infty$ for $\alpha_0(b)<\alpha$.

We have obtained the following for the global existence.
\begin{theorem}[Global existence for oscillating data]
\label{global}

Let $N\geq 1$, $u_0\in H^s(\R^N)$, $s=0,1$ and $K$ be a complex valued function   satisfying the condition $(K_1)$.  Let $1\leq q \leq \infty$, $\alpha_0(b)$, $a$ and $\varrho$ be defined respectively by \eqref{alpha0} and \eqref{aP}. Assume further that\newline
$$s=0\mbox{ and } \alpha_0(b)<\alpha\leq {4-2b\over N},$$
or $$s=1, \; N\geq 4,\; K \mbox{ satisfies } (K_2) \mbox{ and }\alpha_0(b)<\alpha < {4-2b\over N-2}.$$
Then there exists $\varepsilon>0$ such that if
\begin{equation}
\label{globalcondi}
\|e^{it\Delta}u_0\|_{L^a(0,\infty; L^{\varrho,q}(\R^N))}\leq \varepsilon,
 \end{equation} the maximal solution $u$ of \eqref{INLS} with initial value $u_0$  is positively global.  Moreover, $u\in L^a\left(0,\infty; L^{\varrho,q}(\R^N)\right)\cap L^r\left(0,\infty; W_2^{s,p}(\R^N)\right),$ for every admissible pair $(r,p)$ and there exists a constant $C>0$ such that the following estimates hold
 \begin{equation}
 \label{estimatinglobalerapid}
 \|u\|_{ L^{a}(0,\infty;L^{\varrho,q}(\R^N))}\leq 2\|e^{it\Delta}u_0\|_{L^a(0,\infty; L^{\varrho,q}(\R^N))},
 \end{equation}
\begin{equation}
 \label{estimatinglobalestrit}
 \|u\|_{ L^{r}(0,\infty;W_2^{s,p}(\R^N))}\leq  C\|u_0\|_{H^s(\R^N)}.
 \end{equation}
\end{theorem}
Examples of initial data satisfying \eqref{globalcondi} are given in Corollaries \ref{exempleinitialdata} and \ref{occilating} below.  A class of initial data giving rise to blowing up solutions only for negative time is given in  Corollary \ref{occilating}.

\begin{remark}$ ${\rm
\begin{itemize}
\item[1)]  For the well-posedness of the equation \eqref{INLS} with  regularity in Lorentz spaces, see  Theorems 1.2, 1.3 and Remark 1.1 in \cite{AT-1}.

\item[2)] The typical example of potential satisfying the hypothesis of the previous theorem  is $K(x)=\mu |x|^{-b}$, $\mu\in \mathbb{C}$.

\item[3)] Our proof is valid for $b=0$. In particular the case $s=0$ is new. On the other hand, the previous result extends the case $b=0$, $s=1$ and $q=\varrho$ considered in \cite[Proposition 2.4, p. 82]{CW3}.
\item[4)] The conditions $ \alpha >\alpha_0(b)$ and $\varrho <N$ needed for the Sobolev embedding require $N\geq 4$ for $s=1$. The case $N=3$  will be done in the forthcoming work \cite{AT-3}.
\item[5)] If we replace $(0,\infty)$ in \eqref{globalcondi} by an interval $I$, the conclusion of the previous result holds on $I$ instead of $(0,\infty)$. In particular if $I=\R$ we get global existence of the solution for negative and positive time.
\item[6)] The previous result is new for $L^2$-subcritical $\alpha$ with smallness condition for one non-admissible pair. More precisely, let ${\mathcal{A}_{s_c}}$ be the set of $\dot{H}^{s_c}(\R^N)-$admissible pair, defined by
    $${\mathcal{A}_{s_c}}=\{(r,p),\; {2\over r}={N\over 2}-{N\over p}-s_c\}$$  where $s_c={N\over 2}-{2-b\over \alpha}.$ A smallness condition on $\sup_{(r,p)\in {\mathcal{A}_{s_c}}}\|e^{it\Delta}u_0\|_{L^r(\R; L^{p})},$ is imposed in \cite[Corollary 1.12, p. 253]{Guzman} to obtain global existence for solutions of \eqref{INLS}. See also \cite[Proposition 4.5, p. 4189]{FarahGuzman} for $N=3.$ Here our condition combined with the above remark, is less restrictive.  In fat, we consider only one $\dot{H}^{s_c}(\R^N)-$admissible pair which is $(a,\varrho).$   Note that also, unlike \cite{Guzman} where $\alpha>{4-2b\over N}$, here we reach the values $\alpha > \alpha_0(b).$
\item[7)] Using similar argument as in \cite[Theorem 6.2.1, p. 165]{Cazenave}, one can prove that if $${4-2b\over N}\leq\alpha<{4-2b\over N-2}$$ and $u_0\in H^1(\R^N)$ sufficiently small, then the corresponding solution $u$ of   \eqref{INLS} is positively and negatively global. Moreover,  $u\in L^r\left(\R; W_2^{1,p}(\R^N)\right)$, for every admissible pair $(r,p).$ In fact for this result we need $$0<{1\over a}\leq {1\over \varsigma},$$ where $\varsigma$ is such that $(\varsigma, \varrho)$ is an admissible pair, that is
    \begin{equation}
    \label{varsigma}
    \varsigma={4(\alpha+2)\over N\alpha+2b}.
    \end{equation}
    The first inequality follows since $\alpha$ is subcritical. In order that the second inequality is verified we need  $N\alpha^2+2(N-2+b)\alpha-8+4b \geq 0,$ that is  $\alpha\geq {4-2b\over N}$.

For the above values of $\alpha$, the condition \eqref{globalcondi} is less restrictive than the smallness condition in $H^1(\R^N)$. This follows by the inequality
$$\|e^{it\Delta}u_0\|_{L^a(\mathbb{R}; L^{\varrho,q}(\R^N))}\lesssim \|u_0\|_{H^1(\R^N)}$$
which can be obtained by interpolation and the Sobolev-Lorentz embedding.
\item[8)] Recall that for the critical case $\alpha={4-2b\over N-2s},$ the global existence is established in \cite[Theorem 1 Part (v)]{AT-1},  for initial data $u_0\in H^s(\R^N),$ such that $\|(-\Delta)^{\frac{s}{2}} u_0\|_{L^2(\R^N)}$ is sufficiently small.

\item[9)] If $\alpha\geq {4-2b\over N}$, then the rapidly decay property $u\in L^a\left(0,\infty; L^{\varrho,2}(\R^N)\right)$ follows from the fact that $u\in L^r\left(0,\infty; W_2^{1,p}(\R^N)\right)$, for every admissible pair $(r,p)$. Indeed, by the Sobolev-Lorentz embedding theorem, $u\in L^\varsigma\left(0,\infty; L^{\varrho,2}(\R^N)\right)\cap L^\infty\left(0,\infty; L^{\varrho,2}(\R^N)\right).$  Since in this case $\varsigma\leq a$, then the result follows  by interpolation argument. While for $\alpha<{4-2b\over N}$ the property $u\in L^a\left(0,\infty; L^{\varrho,2}(\R^N)\right)$ gives a more precise decay.
\item[10)] The value of $\varrho$ allows the
map $f \to |x|^{-b}|f|^{\alpha}f$ to apply the Lorentz space $L^{\varrho,q}$ into  $L^{\varrho ',\frac{q}{\alpha+1}}$, $q\geq \alpha+1$. For this choice of $\varrho$, if $f\in H^1 $ so $f$ belongs to $ L^{\varrho,\alpha+2}$ and under the condition $(K_1)$, the energy
 \begin{equation}
\label{energy}
E(f)={1\over 2}\int_{\R^N}|\nabla f(x)|^2dx+{1\over \alpha+2}\int_{\R^N}K(x)|f(x)|^{\alpha+2}dx
\end{equation}
is well defined.
\item[11)] The value of $a$ is determined by scaling argument. In fact, if $u$ is a solution of \eqref{INLS} with $K(x)=|x|^{-b}$, then for all $\lambda>0,$ $u_\lambda$ is also a solution of  \eqref{INLS}, where $u_\lambda(t,x)=\lambda^{{2-b\over \alpha}}u(\lambda^2 t,\lambda x).$ The value of $a$ is determined so that $\|e^{it\Delta}u_\lambda(0)\|_{L^a(0,\infty; L^{\varrho,q})}$ is independent of $\lambda.$ In fact, let $D_\lambda$ be the dilation operator defined by $D_\lambda(f)(x)=f(\lambda x),\; \lambda>0.$ It is known that
    $$D_\lambda{\rm e}^{i\lambda^2 t\Delta} = {\rm e}^{it\Delta}D_\lambda.$$ See \cite[equality (3.2), p. 259]{CW2}. Hence we have
\begin{eqnarray*}
\|e^{it\Delta}u_\lambda(0)\|_{L^a(0,\infty; L^{\varrho,q})}&=&\|e^{it\Delta}\lambda^{{2-b\over \alpha}}D_\lambda u_0\|_{L^a(0,\infty; L^{\varrho,q})}\\ &=& \lambda^{{2-b\over \alpha}-{N\over \varrho}-{2\over a}}\| e^{i t\Delta} u_0\|_{L^a(0,\infty; L^{\varrho,q})}.
\end{eqnarray*}
The value of $a$ gives
${2-b\over \alpha}-{2\over a}-{N\over \varrho}=0.$

\end{itemize}}
\end{remark}

We now establish a scattering criterion in $L^2$. We have obtained the following.
\begin{theorem}[$(L^2,L^2)$-scattering criterion]
\label{globalstritchartzu}
Let $N\geq 1$, $u_0\in L^2(\R^N)$, $K$ be a complex valued function   satisfying the condition $(K_1)$ and $$ 0<\alpha\leq {4-2b\over N}.$$ Let $a$, $\varrho$ and $\varsigma$ be given by \eqref{aP} and \eqref{varsigma}.
Let $u\in C([0,\infty),L^2(\R^N))\cap L^\varsigma_{loc}(0,\infty; ,L^\varrho (\R^N))$ be a global solution of \eqref{INLS}. If $|K|^{{1\over \alpha+2}}u\in L^{ a}(0,\infty;L^{{\alpha+2},\infty}(\R^N))$ then  $u\in L^{r}(0,\infty;L^{p,2}(\R^N)),$ for any admissible pair $(r,p)$ and $u$ scatters in $L^2(\R^N)$, that is there exists $\varphi^+\in L^2(\R^N)$ such that
$$\lim_{t\to\infty}\|u(t)-e^{it\Delta}\varphi^+\|_{L^2(\R^N)}=0.$$
 Similar statements hold for negative time.
\end{theorem}

We now turn to establish a criterion for the scattering in $H^1$ and in $\Sigma$, where
\begin{equation}
\label{sigma}
\Sigma:=\{\varphi\in H^1(\R^N),\; |\cdot|\varphi(\cdot)\in L^2(\R^N)\}.
\end{equation}
We will need the following hypothesis on $K$,
$$(K_3)\quad |\nabla K(x)|\lesssim |x|^{-1-\frac{2b}{\alpha+2}}|K(x)|^{\alpha /(\alpha+2)},$$
 for $ x\in \R^{N}\setminus \{0\}$.
Define $\tilde{K}(x)=|x|^{-b}$ if $K$ verifies $(K_2)$ and $\tilde{K}(x)=K(x)$ if $K$ verifies $(K_3)$.
We have obtained the following result.
\begin{theorem}[$(H^1,H^1)$ and $(\Sigma,\Sigma)$-Scattering criterion]
\label{globalstritchartzuw}
Assume $N\geq 4$, $u_0\in H^1(\R^N)$, $0<b<2$, $$ 0<\alpha< {4-2b\over N-2}$$ and $K$ be a complex valued function   satisfying the condition $(K_1)-(K_2)$ or $(K_1)$ and $(K_3)$.
 Let $a$ be given by \eqref{aP} and $u\in C([0,\infty),H^1(\R^N))$ be a global solution of \eqref{INLS}. If $|\tilde{K}(x)|^{{1/ (\alpha+2)}}u\in L^{ a}(0,\infty;L^{{\alpha+2},\infty}(\R^N))$ then  $u\in L^{r}(0,\infty;W_2^{1,p}(\R^N)),$ for any admissible pair $(r,p)$ and $u$ scatters in $H^1(\R^N).$ Moreover, if $u_0\in \Sigma$ then $u\in C([0,\infty),\Sigma)$  and it scatters in $\Sigma$ that is there exists $\varphi^+\in \Sigma$ such that
$$\lim_{t\to\infty}\|e^{-it\Delta}u(t)-\varphi^+\|_{\Sigma}=0.$$
Similar statements hold for negative time.
\end{theorem}
The  previous results give the following.
\begin{corollary}[Scattering for oscillating data]\label{Scattering for oscillating solutions}
Assume the hypothesis of Theorem \ref{global}. Let $\varepsilon$ be given by Theorem \ref{global} and $u_0\in H^s(\R^N)$ satisfying \eqref{globalcondi}. Then the corresponding solution $u$ of \eqref{INLS}-\eqref{InitialINLS} scatters in $H^s(\R^N)$ as $t\to \infty.$ Moreover, if $u_0\in \Sigma$   then $u$ scatters in $\Sigma$ as $t\to \infty.$

\end{corollary}
 We have the following result which shows the impact of the  decay of the potential on the range of $\alpha$ allowing the scattering.

\begin{corollary}(Scattering for potentials with different powers for singularity and decay)\label{EffectDecayK} Let $s=0$ or $s=1$. Assume that $N\geq 1$ if $s=0$ and $N\geq 4$ if $s=1$. Let $0\leq b_1<\min(2,N)$, $b_2> b_1$ and $K$ satisfying $|K(x)|\lesssim |x|^{-b_1}(1+|x|^2)^{-(b_2-b_1)/2}$. For $s=1$ we suppose further that $|\nabla K(x)|\lesssim |x|^{-b_1-1}(1+|x|^2)^{-(b_2-b_1)/2}$ if $b_1>0$ and $|\nabla K(x)|\lesssim (1+|x|^2)^{-(b_2+1)/2}$ if $b_1=0$ .  Let $\alpha_0(\min (2,b_2))<\alpha <\frac{4-2b_1}{N-2s}$ and $b$ such that
$$\max \left(b_1,\frac{4-N\alpha^2-\alpha(N-2)}{2(\alpha+1)}\right)\leq b\leq \min \left(b_2,\frac{4-(N-2s)\alpha}{2}\right).$$ Let $u_0\in H^s$. Then there exists $\varepsilon =\varepsilon (b_1,b_2,\alpha)>0,$ such that if
\begin{equation*}
\|e^{it\Delta}u_0\|_{L^a(0,\infty; L^{\varrho,q}(\R^N))}\leq \varepsilon,
 \end{equation*} where $a,\varrho$ are given by \eqref{aP},  the solution of \eqref{INLS}-\eqref{InitialINLS} is global and scatters in $H^s(\mathbb{R}^N)$. Moreover, if $u_0\in \Sigma$   then $u$ scatters in $\Sigma$ as $t\to \infty.$

In particular, if the potential $K$ is regular and has some decay at infinity, for example $K(x)=(1+|x|^2)^{-b_2/2}$ that is $b_1=0$ and $b_2>0$, then for $\alpha_0(\min (b_2,2))<\alpha <\frac{4}{N-2}$ we have scattering for small initial data in $\Sigma$. In the special case where $b_2\geq 2$, then the scattering holds for any $0<\alpha <\frac{4}{N-2}$.
 \end{corollary}

It is known that for the defocusing case ($K\geq 0$)  the solutions of \eqref{INLS} are global. In order to study the scattering for this case we establish decay estimates in Lorentz spaces for  initial data in $\Sigma$. We need the following condition on $K$.
$$(K_4)\quad x.\nabla K(x)\leq -bK(x),$$
 for $ x\in \R^{N}\setminus \{0\}$.
We have obtained the following.
\begin{theorem}[Decay estimates]
\label{Decay Estimates}
Let $N\geq 3$, $0<b<2$,
 $$0<\alpha <{4-2b\over N-2}$$
and $K(x)\geq 0$, $K\in C^1(\R^N\setminus \{0\})$ satisfying  $(K_1)$,  $(K_3)$  and $(K_4)$. Let   $u_0\in \Sigma$  and $u\in C(\R,H^1(\R^N))$ be the solution of \eqref{INLS},  with initial $u_0$. Then $u\in C(\R,\Sigma)$ and the following hold.
\begin{itemize}
\item[(i)]If $\alpha\geq (4-2b)/N,$  then for every $2\leq p\leq 2N/(N-2)$  there exists a positive constant $C>0$ such that  for every $t\not=0,$ we have
\begin{equation}
\label{estimatelpq}
\|u(t)\|_{L^{p,2}(\R^N)}\leq C\left(\|u_0\|_{L^2(\R^N)}+\|xu_0\|_{L^2(\R^N)}\right)|t|^{-N\left({1\over 2}-{1\over {p}}\right)}.
\end{equation}

\item[(ii)] If $\alpha<(4-2b)/N,$ then  for every $p$ satisfying
$${N-2\over 2N}+\bar{\sigma}<{1\over p}<{1\over 2}+\bar{\sigma}$$
with $$\bar{\sigma}={b\over N(\alpha+2)} $$ there exists a positive constant $C>0,$ such that
\begin{equation}
\label{decy-blpq2}
\|K^{1/(\alpha+2)}u(t)\|_{L^{p,1}(\R^N)}\leq C|t|^{-N\left({1\over 2}-{1\over p}+\bar{\sigma}\right)(1-\delta)},
\; \mbox{ for every}\; t\not=0,
\end{equation}
where
$$\delta=\begin{cases}
0, & \; \mbox{ if } \; p\leq \alpha+2,\\
{N(\alpha+2)(p-\alpha-2)(4-2b-N\alpha)\over \left[N(\alpha+2)(p-2)+2bp\right]\left[4-2b-\alpha(N-2)\right]},  & \; \mbox{ if } \; p>\alpha+2.\\
\end{cases}$$
\end{itemize}
\end{theorem}

We have the following remarks.
\begin{remark}$ \, ${\rm
\begin{itemize}
\item[1)] The previous estimates for $b=0$ are  known in the Lebesgue spaces (see \cite{Cazenave}). Taking into account that Lorentz spaces are increasing with respect to the second indices, our estimates are more precise.
\item[2)] Other estimates are known for $b>0.$ See \cite[Theorem 1.3, p. 4]{Dinh}. For $\alpha\geq (4-2b)/N,$ our estimate, given in $L^{p,2}(\R^N),$ is more precise than the one of \cite{Dinh} established in $L^{p}(\R^N).$ For $\alpha<(4-2b)/N$ and $p\not=\alpha+2$ our result is new. In fact, even by combining \cite[estimate (1.15), p. 4]{Dinh}, with the H\"older inequality in Lorentz spaces, we get the estimate:
     $$\||\cdot|^{-b/(\alpha+2)}u(t)\|_{L^{p,q_1}(\R^N)}\leq C|t|^{-N\left({1\over 2}-{1\over p}+\bar{\sigma}\right){2b+N\alpha\over 4}},
\; \mbox{ for every}\; t\not=0,$$
with $${1\over q_1}={1\over p}-\bar{\sigma},$$  which gives a slower decay rate than \eqref{decy-blpq2} in a larger space ($q_1>q$).
\item[3)] In (i) we can deduce an estimate with weight as for (ii). In fact, let $\alpha\geq (4-2b)/N,$ and  $p$ be such that
$${N-2\over 2N}+\bar{\sigma}<{1\over p}<{1\over 2}+\bar{\sigma};\; \mbox{where}\; \bar{\sigma}={b\over N(\alpha+2)}.$$
By using H\"older's inequality in Lorentz spaces and (i), we deduce the existence of a constant $C>0,$ such that
\begin{equation}
\label{estimatex-blpq}
\left\||\cdot|^{-{b\over \alpha+2}}u(t)\right\|_{L^{p,2}(\R^N)}\leq C|t|^{-N\left({1\over 2}-{1\over p}+\bar{\sigma}\right)},
\; \mbox{ for every}\; t\not=0.
\end{equation}
\end{itemize}}
\end{remark}
Combining Theorems \ref{globalstritchartzuw} and \ref{Decay Estimates}, we deduce the following result.

\begin{corollary}[$(\Sigma,\Sigma)-$Scattering  for the  defocusing case]
\label{scatering}
Let $N\geq 4$ and $K(x)\geq 0$, $K\in C^1(\R^N\setminus \{0\})$ satisfying  $(K_1)$,  $(K_3)$  and $(K_4)$. Assume that    $$\alpha_0(b)<\alpha<{4-2b\over N-2},$$ where $\alpha_0(b)$ is defined by \eqref{alpha0}. Let $u_0\in \Sigma$ and $u$ be the unique solution of \eqref{INLS} with initial data $u_0.$ Then $u$ scatters in $\Sigma$ for positive and negative time.
\end{corollary}
\begin{remark}
\label{remscatering}
{\rm $\,$
\begin{itemize}
\item[1)] The previous statement holds for  $K(x)=|x|^{-b}$. For this particular case our result is new if $\alpha_0(b)<\alpha<{4-2b\over N}.$ The case  ${4-2b\over N}\leq \alpha<{4-2b\over N-2}$ is known ( see \cite[Theorem 1.4]{Dinh}). Owing to the decay estimates, we provide a simpler proof.
\item[2)] The cases $N=2,\; N=3$ and $\alpha=\alpha_0(b)$ will be treated in forthcoming papers.
\end{itemize}}
\end{remark}

The rest of the paper is organized as follows. Section 2 is devoted to some preliminaries. In particular we establish Strichartz estimates for non-admissible pairs in Lorentz spaces. See Proposition \ref{nonadmissiblestrichartz}. In Section 3 we prove the global existence for oscillating initial values. The proofs of the scattering criteria are given in Section 4. In Section 5 we establish the decay estimates for initial data in $\Sigma$. Section 6 is devoted to the proofs of the  scattering results. Finally, in Section 7 we give a scattering criterion and a scattering result in $H^1$.

In the sequel, a functional space on $\R^N,\; X(\R^N)$ will be  denoted simply by $X.$ The notation $A\lesssim B$ for positive numbers $A$ and $B$, means that there exists a positive constant $C$ such that $A\leq CB.$ If $A\lesssim B$ and $B\lesssim A,$ we write $A\sim B.$  $C$ will denotes a constant which may be different at different places. For $p\geq 1$, $p'=p/(p-1)$ denotes its conjugate exponent.
\section{Preliminaries}
In this section we give some preliminaries which will be needed for the proofs. For the definitions and properties of Lorentz spaces see \cite{Lemarie} and references therein.

We recall the following interpolation inequality in Lorentz spaces
\begin{equation}
\label{interpolation}
\|f\|_{L^{p,q}}\leq C\|f\|_{L^{p_1,q_1}}^\theta\|f\|_{L^{p_2,q_2}}^{1-\theta},
\end{equation}
where $1<p,\;p_1,\; p_2<\infty,$  $1\leq q,\; q_1,\; q_2\leq\infty,\; \theta\in (0,1)$ and
$${1\over p}={\theta\over p_1}+{1-\theta\over p_2},\; {1\over q}\leq {\theta\over q_1}+{1-\theta\over q_2}.$$

We define the Sobolev-Lorentz spaces (See \cite[page 571]{HYZ}) as follows
$$W^{s,p}_q(\R^N)=\{f\in \mathcal{S}'(\R^N),\; (I-\Delta)^{s/2}f\in L^{p,q}(\R^N)\},$$
$$\dot{W}^{s,p}_q(\R^N)=\{f\in \mathcal{S}'(\R^N),\; (-\Delta)^{s/2}f\in L^{p,q}(\R^N)\},$$
for $s\geq 0,\; 1< p<\infty,\; 1\leq q\leq \infty.$

We recall the homogenous Sobolev-Lorentz embedding (See \cite[theorem 2.4 (iii), p. 20]{Lemarie}) : $\dot{W}^{s,p}_q(\R^N)\hookrightarrow L^{\tilde{p},q}(\R^N),$ where $1<p<\infty,\; 1\leq q\leq \infty,0<s<\frac{N}{p}$ and
$$\frac{1}{\tilde{p}}=\frac{1}{p}-\frac{s}{N}.$$ That is  there exists a constant $C>0$ such that
\begin{eqnarray}\label{Hom-Sobolev-Inject}
\|f\|_{L^{\tilde{p},q}}\leq C\|(-\Delta)^{s/2}f\|_{L^{p,q}},\; f\in \;\dot{W}^{s,p}_q(\R^N).
\end{eqnarray}

By the well known Sobolev embedding $H^s(\R^N) \hookrightarrow L^{p}(\R^N)$   and interpolation, we  have  the following  $$H^s(\R^N) \hookrightarrow L^{p,2}(\R^N),\; s\geq 0,\; {1\over 2}-{s\over N}\leq {1\over p}\leq {1\over 2}, \; p<\infty.$$

Finally,  recall the Gagliardo-Nirenberg inequality in the Lorentz spaces
\begin{equation}
\label{gagliardoNiren1}
\|f\|_{L^{p,q}}\leq C \|(-\Delta)^{s/2}f\|_{L^{p_1,q_1}}^\theta \|f\|_{L^{p_2,q_2}}^{1-\theta},
\end{equation}
where $1<p,\; p_2<\infty,$  $1< q,\; q_1,\; q_2<\infty,\; 0<s<N,\; 1<p_1<N/s,\; 0<\theta<1,$ and
$${1\over p}={\theta\over p_1}-{\theta s\over N}+{1-\theta\over p_2},\;{1\over q}\leq  {\theta\over q_1}+{1-\theta\over q_2}.$$
See \cite[Theorem 2.1, p. 571]{HYZ}.

The following result is well known.
\begin{proposition}\label{Prop-Disp-Loren}
Let $N\geq 1$ be an integer. Let $2<p<\infty$ and $1\leq q_1\leq q_2\leq \infty$. There exists a constant $C>0$ such that for every $\varphi\in L^{p',q_1}$ we have
\begin{equation}
\Vert e^{it\Delta}\varphi\Vert_{L^{p,q_2}}\leq C |t|^{-\frac{N}{2}(1-\frac{2}{p})}\Vert \varphi\Vert_{L^{p',q_1}},\label{Est-Disp-Loren}
\end{equation}
for all $t\neq 0$.
\end{proposition}


The following theorem is a key tool for our work.
\begin{proposition}\cite[Theorem 10.1]{KT}
\label{Strichartzestimhominhom}
(i) Let $(r,p)$ be an admissible pair with $p<\infty$. Then there exists a constant $C>0$, such that
\begin{eqnarray}
\label{Strichartzestimhom}
\left\Vert e^{it\Delta}\varphi\right\Vert_{L^{r}(\mathbb{R}, L^{p,2}(\mathbb{R}^N))}& \leq & C \Vert \varphi\Vert_{L^{2}(\mathbb{R}^N)},
\end{eqnarray}
for every $\varphi\in L^2$.

(ii) Let $(r_i,p_i)$; $i=1,2$ be two admissible pairs with $p_i<\infty$. Then there exists a constant $C>0$, such that
\begin{eqnarray}
\label{Strichartzestim}
\left\Vert \int_0^t e^{i(t-s)\Delta}f(s,.)ds\right\Vert_{L^{r_1}(\mathbb{R},L^{p_1,2}(\mathbb{R}^N))}& \leq & C \Vert f\Vert_{L^{r_2'}(\mathbb{R},L^{p_2',2}(\mathbb{R}^N))},
\end{eqnarray}
for every $f\in L^{r_2'}(\mathbb{R},L^{p_2',2}(\mathbb{R}^N))$.
\end{proposition}
The  following Strichartz estimates treat the case of non-admissible pairs. When the second index of the Lorentz space is equal to $2$ such estimates are well known (See \cite{taggart,Vilela}). Here we consider the general case which will be used to prove the global existence.
\begin{proposition}
\label{nonadmissiblestrichartz}
Assume that $N\geq 1$. Let $0<T\leq \infty.$ Let $(r,p)$ be an admissible pair with $p<\infty$, $r>2,$ and  $1\leq q \leq \infty.$ Fix $\sigma>r/2$ and define $\tilde{\sigma}$ by
$${1\over \tilde{\sigma}}+{1\over \sigma}={2\over r}.$$
Let $$\mathcal{F}(f)(t)=\int_0^te^{i(t-s)\Delta}f(s,\cdot)ds,\; t\in [0,T).$$
If $f\in L^{ \tilde{\sigma}'}(0,T;L^{p',q}(\R^N)),$ then $\mathcal{F}(f)\in L^{\sigma}(0,T;L^{p,q}(\R^N)).$ Moreover there exists a constant $C=C(N,p,q,\sigma)>0$ such that
$$\|\mathcal{F}(f)\|_{ L^{ \sigma}(0,T;L^{p,q}(\R^N))}\leq \|f\|_{ L^{\tilde{\sigma}'}(0,T;L^{p',q}(\R^N))}$$ for every $f\in L^{ \tilde{\sigma}'}(0,T;L^{p',q}(\R^N)).$
\end{proposition}
\begin{proof}
For $t\in (0,T),$ we write
\begin{eqnarray*}
\mathcal{F}(f)(t)&=&\int_\R\chi_{[0,t]}(s)e^{i(t-s)\Delta}f(s,\cdot)ds,
\end{eqnarray*}
where $\chi_{[0,t]}$ is the characteristic function of the interval $[0,t].$ Using the dispersive estimate \eqref{Est-Disp-Loren}, we get
\begin{eqnarray*}
\|\mathcal{F}(f)(t)\|_{ L_x^{p,q}(\R^N)}&\leq &\int_\R\chi_{[0,t]}(s)\|e^{i(t-s)\Delta}f(s,\cdot)\|_{ L^{p,q}(\R^N)}ds\\
&\leq &\int_\R\chi_{[0,t]}(s)|t-s|^{-{2\over r}}\|f(s,\cdot)\|_{ L^{p',q}(\R^N)}ds\\
&= &\left(|s|^{-{2\over r}}\star\|\chi_{[0,T]}(s)f(s,\cdot)\|_{ L^{p',q}(\R^N)}\right)(t).
\end{eqnarray*}
Using Young's inequality in Lorentz space in time and since $1+{1\over \sigma}={2\over r}+{1\over\tilde{\sigma}'}$ and $r>2$ that is ${1\over \sigma}\leq {1\over \tilde{\sigma}'},$ we get
\begin{eqnarray*}
\|\mathcal{F}(f)\|_{ L^{\sigma}(0,T;L^{p,q}(\R^N)}&\leq &\left\| \left(|s|^{-{2\over r}}\star\|\chi_{[0,T]}(s)f(s,\cdot)\|_{ L^{p',q}(\R^N)}\right)(t)\right\|_{L^{\sigma}(0,T)}\\
&\leq &C\| |\cdot|^{-{2\over r}}\|_{L^{r/2,\infty}(\R)}\|\chi_{[0,T]}f\|_{L^{ \tilde{\sigma}',\sigma}(\R;L^{p',q}(\R^N))}\\
&\leq &\|f\|_{L^{ \tilde{\sigma}'}(0,T;L^{p',q}(\R^N))},
\end{eqnarray*}
This completes the proof of the lemma.
\end{proof}
\section{Global existence for oscillating initial values}
In this section we prove the global existence of solution to  equation \eqref{INLS}, that is  Theorem \ref{global}. We first establish some preliminaries results.
\begin{lemma}
\label{nonlinearnonadmissibleNgeq3}
Assume that $N\geq 1$. Let $K$ be a complex valued function   satisfying the condition $(K_1)$ and $0<\alpha<(4-2b)/(N-2)_+.$ Let $\alpha_0(b)$, $a$, $\varrho$  and $\varsigma$ be given  respectively by  \eqref{alpha0}, \eqref{aP}  and \eqref{varsigma}.  Then we have
\begin{itemize}
\item[(i)] $a>\varsigma/2$ if and only if $a>\alpha+1$ if and only if $\alpha>\alpha_0(b).$
\end{itemize}
For  $0<T\leq \infty,$ we have the following.
\begin{itemize}
\item[(ii)] If $u\in L^{a}(0,T;L^{\varrho,\infty}(\R^N))$ and $v\in L^{\varsigma}(0,T;L^{\varrho,2}(\R^N))$ then for any admissible pair $(r,p)$,  $\mathcal{F}(K|u|^\alpha v)\in L^{r}(0,T;L^{p,2}(\R^N)).$ Moreover, there exists a constant $C>0$ independent of $T$ such that
\begin{equation}
\label{similar 2.4l2}\|\mathcal{F}(K |u|^\alpha v)\|_{ L^{r}(0,T;L^{p,2}(\R^N))}\leq C \|u\|^{\alpha}_{ L^{a}(0,T;L^{\varrho,\infty}(\R^N))}\|v\|_{ L^{\varsigma}(0,T;L^{\varrho,2}(\R^N))}.
\end{equation}

\item[(iii)] If $\alpha>\alpha_0(b)$ and $u\in L^{a}(0,T;L^{\varrho,q}(\R^N)),\; 1\leq q\leq \infty$ then $\mathcal{F}(K |u|^\alpha u)\in L^{ a}(0,T;L^{\varrho,{q\over \alpha+1}}(\R^N)).$ Moreover there exists a constant $C>0$ independent of $T$ such that
\begin{equation}
\label{similar 2.3}\|\mathcal{F}(K |u|^\alpha u)\|_{ L^{a}\left(0,T;L^{\varrho,{q\over \alpha+1}}(\R^N)\right)}\leq C \|u\|^{\alpha+1}_{ L^{a}(0,T;L^{\varrho,q}(\R^N))}.
\end{equation}
\end{itemize}
\end{lemma}
\begin{proof} The proof of Part (i) follows by simple calculations using \eqref{aP} and  \eqref{alpha0}.

(ii) Using the Strichartz estimates  \eqref{Strichartzestim} and H\"older's inequality in Lorentz spaces, we get
\begin{eqnarray*}
\|\mathcal{F}(K |u|^\alpha v)\|_{ L^{r}(0,T;L^{p,2}(\R^N))} &\leq &C\|K |u|^\alpha v\|_{ L^{\varsigma'}(0,T;L^{\varrho',2}(\R^N))}\\&\leq &C\||\cdot|^{-b}|u|^\alpha v\|_{ L^{\varsigma'}(0,T;L^{\varrho',2}(\R^N))}\\
&\leq & C\||u|^\alpha v\|_{ L^{\varsigma'}(0,T;L^{\bar{\varrho},2}(\R^N))}
\\&\leq & C\||u|^\alpha\|_{ L^{{a\over \alpha}}(0,T;L^{{\varrho\over \alpha},\infty }(\R^N))}\|v\|_{ L^{\varsigma}(0,T;L^{\varrho,2}(\R^N))}\\&\leq & C\|u\|^\alpha_{ L^{{a}}(0,T;L^{{\varrho},\infty}(\R^N))}\|v\|_{ L^{\varsigma}(0,T;L^{\varrho,2}(\R^N))},
\end{eqnarray*}
where we have used
$${\alpha\over a}=1-{2\over \varsigma},\; {1\over \varrho'}={b\over N}+{1\over \bar{\varrho}}={b\over N}+{\alpha+1\over \varrho}.$$

(iii) Using Proposition  \ref{nonadmissiblestrichartz}  with $f=K |u|^\alpha u$ and the H\"older inequality in Lorentz spaces, we get
\begin{eqnarray*}
\|\mathcal{F}(K |u|^\alpha u)\|_{ L^{ a}\left(0,T;L^{\varrho,{q\over \alpha+1}}(\R^N)\right)} &\leq &\|K |u|^\alpha u\|_{ L^{\tilde{a}'}(0,T;L^{\varrho',{q\over \alpha+1}}(\R^N)}\\&\leq &\||\cdot|^{-b}|u|^\alpha u\|_{ L^{\tilde{a}'}(0,T;L^{\varrho',{q\over \alpha+1}}(\R^N)}\\
&\leq & C\|||u|^\alpha u\|_{ L^{\tilde{a}'}(0,T;L^{\bar{\varrho},{q\over \alpha+1}}(\R^N)},
\end{eqnarray*}
where ${1\over \tilde{a}}+{1\over a}={2\over\varsigma}$ and ${1\over \bar{\varrho}}={1\over \varrho'}-{b\over N}=1-{1\over \varrho}-{b\over N}<1.$ Then since, $\bar{\varrho}(\alpha+1)=\varrho$ and $\tilde{a}'(\alpha+1)=a,$ we get
\begin{eqnarray*}
\|\mathcal{F}(K|u|^\alpha u)\|_{ L^{ a}\left(0,T;L^{\varrho,{q\over \alpha+1}}(\R^N)\right)} &\leq &  C\|u\|^{\alpha+1}_{ L^{\tilde{a}'(\alpha+1)}(0,T;L^{\bar{\varrho}(\alpha+1),q}(\R^N))}\\&\leq &  C\|u\|^{\alpha+1}_{ L^{a}(0,T;L^{\varrho,q}(\R^N))}.
\end{eqnarray*}

This completes the proof of Lemma \ref{nonlinearnonadmissibleNgeq3}.
\end{proof}
\begin{lemma}
\label{nonlinearnonadmissible}
Assume that $N\geq 4 $. Let $K$ be a complex valued function   satisfying the conditions $(K_1)$ and $(K_2)$. Let $\alpha_0(b)$, $a$, $\varrho$  and $\varsigma$ be given  respectively by  \eqref{alpha0}, \eqref{aP}  and \eqref{varsigma}. Suppose that $ \alpha_0(b)<\alpha<(4-2b)/(N-2).$ If $u\in L^{a}(0,T;L^{\varrho,\infty}(\R^N))\cap L^{\varsigma}(0,T;W_2^{1,\varrho}(\R^N)),$ for  $0<T\leq \infty,$ then for any admissible pair $(r,p),$  $\mathcal{F}(\nabla K |u|^\alpha u)\in L^{r}(0,T;L^{p,2}(\R^N)).$ Moreover there exists a constant $C>0$ independent of $T$ such that
\begin{equation}
\label{similar 2.4gradienxb}
\|\mathcal{F}(\nabla K |u|^\alpha u)\|_{ L^{r}(0,T;L^{p,2}(\R^N))}\leq C \|u\|^{\alpha}_{ L^{a}(0,T;L^{\varrho,\infty}(\R^N))}\|\nabla u\|_{ L^{\varsigma}(0,T;L^{\varrho,2}(\R^N))}.
\end{equation}
Furthermore, $\mathcal{F}(K |u|^\alpha u)\in L^{r}(0,T;W_2^{1,p}(\R^N))$ and there exists a constant $C>0$ independent of $T$ such that
\begin{equation}
\label{similar 2.4}\|\mathcal{F}(K|u|^\alpha u)\|_{ L^{r}(0,T;W_2^{1,p}(\R^N))}\leq C \|u\|^{\alpha}_{ L^{a}(0,T;L^{\varrho,\infty}(\R^N))}\|u\|_{ L^{\varsigma}(0,T;W_2^{1,\varrho}(\R^N))}.
\end{equation}
\end{lemma}
\begin{proof}
Let $\tilde{\varrho}$ be defined by
$${1\over \tilde{\varrho}}={\alpha\over \varrho}+{N-\varrho\over N\varrho}= {\alpha+1\over  \varrho}-{1\over N}.$$ Using the Strichartz estimates  \eqref{Strichartzestim}, H\"older's and Sobolev's inequalities in Lorentz spaces, we have
\begin{eqnarray*}
\left\|\int_0^te^{i(t-s)\Delta}(\nabla K|u|^\alpha u(s))ds\right\|_{ L^{r}(0,T;L^{p,2}(\R^N))} &\leq &C\|\nabla K|u|^\alpha u\|_{ L^{\varsigma'}(0,T;L^{\varrho',2}(\R^N))}\\&& \hspace{-5cm} \leq C\||\cdot|^{-b-1}|u|^\alpha u\|_{ L^{\varsigma'}(0,T;L^{\varrho',2}(\R^N))}\\
&& \hspace{-5cm} \leq  C\||u|^\alpha u\|_{ L^{\varsigma'}(0,T;L^{\tilde{\varrho},2}(\R^N))}
\\&& \hspace{-5cm} \leq C\||u|^\alpha\|_{ L^{{a\over \alpha}}(0,T;L^{{\varrho\over \alpha},\infty}(\R^N))}\|u\|_{ L^{\varsigma}(0,T;L^{{N\varrho\over N-\varrho},2}(\R^N))}\\&& \hspace{-5cm} \leq C\|u\|^\alpha_{ L^{{a}}(0,T;L^{{\varrho},\infty}(\R^N))}\|\nabla u\|_{ L^{\varsigma}(0,T;L^{\varrho,2}(\R^N))},
\end{eqnarray*}
where we have used, $\varrho<N,$\; since $N\geq 4$ and
$$0<b<N+1,\; {1\over \varrho'}={b+1\over N}+{1\over \tilde{\varrho}}={b\over N}+{\alpha+1\over \varrho}.$$
$${1\over \varsigma'}={\alpha\over a}+{1\over \varsigma}.$$

The last statements follows by Lemma \ref{nonlinearnonadmissibleNgeq3} and the fact that
\begin{eqnarray*} \nabla(\mathcal{F}(K|u|^\alpha u))(t) & = &\int_0^te^{i(t-s)\Delta}(\nabla(K|u|^\alpha u(s)))ds\\
& = & \int_0^te^{i(t-s)\Delta}(\nabla K|u|^\alpha u(s))ds\\&&+\frac{\alpha+2}{2}\int_0^te^{i(t-s)\Delta}(K|u(s)|^\alpha \nabla u(s))ds \\&&+\frac{\alpha}{2}\int_0^te^{i(t-s)\Delta}(K|u(s)|^{\alpha-2}u^2 \nabla \overline{u}(s))ds.
\end{eqnarray*}
 This completes the proof of Lemma \ref{nonlinearnonadmissible}.
\end{proof}

We now give the proof of the global existence.
\begin{proof}[Proof of Theorem \ref{global}] Let $\varepsilon>0$ to be fixed later. Let $a,\; \varrho$ and $\varsigma$ be given by \eqref{aP}and \eqref{varsigma}. Consider $u_0\in H^s(\R^N)$, $s=0,1$, be such that $\|e^{it\Delta}u_0\|_{L^a(0,\infty; L^{\varrho,q})}\leq \varepsilon$ and $u$ be the maximal solution of \eqref{INLS} with initial data $u_0$  defined on $[0,T_{\max}(u_0))$ with $0<T_{\max}(u_0)\leq \infty$ be the maximal existence time of $u$ (see \cite{AT-1}). Let $\varsigma$ be such that $(\varsigma,\varrho)$ is an  admissible pair. It follows by applying \eqref{similar 2.3} and  \eqref{Strichartzestimhom} together with  \eqref{similar 2.4}  for $s=1$ or \eqref{similar 2.4l2} for $s=0$ on the equation \eqref{intINLS}, that there exists a positive constant $C$ such that for every $T<T_{\max}(u_0)$
\begin{equation}
\label{similar2.6}
\|u\|_{ L^{ a}(0,T;L^{\varrho,q}(\R^N)}\leq \varepsilon +C \|u\|^{\alpha+1}_{ L^{a}(0,T;L^{\varrho,q}(\R^N)},
\end{equation}
\begin{equation}
\label{similar2.7}
\|u\|_{ L^{r}(0,T;W_2^{s,p}(\R^N))}\leq C\|u_0\|_{H^1(\R^N)}+ C \|u\|^{\alpha}_{ L^{a}(0,T;L^{\varrho,q}(\R^N))}\|u\|_{ L^{\varsigma}(0,T;W_2^{s,\varrho}(\R^N))}.
\end{equation}
Choose $\varepsilon$  such that $2^{\alpha+1}C\varepsilon^\alpha<1.$ It follows by \eqref{similar2.6} and classical Gronwall's argument that
\begin{equation}
\label{similar2.9}
\|u\|_{ L^{ a}\left(0,T_{\max}(u_0);L^{\varrho,q}(\R^N)\right)}\leq 2\varepsilon.
\end{equation}
Now \eqref{similar2.7} with $(r,p)=(\varsigma, \varrho)$ and \eqref{similar2.9}, give
\begin{equation}
\label{similar2.10}
\|u\|_{ L^{\varsigma}\left(0,T_{\max}(u_0);W_2^{s,\varrho}(\R^N)\right)}\leq 2 C\|u_0\|_{H^s(\R^N)}.
\end{equation}
By the blow-up alternative (see  \cite[Theorem 1.2]{AT-1}) we conclude that $T_{\max}(u_0)=\infty.$ The other properties satisfied by the solution $u$  follow by \eqref{similar2.7}, \eqref{similar2.9} and \eqref{similar2.10} with $T=T_{\max}(u_0)=\infty.$

We now prove the estimates satisfied by $u.$ We write
$$
\|u\|_{ L^{ a}(0,\infty;L^{\varrho,q}(\R^N))}\leq \|e^{it\Delta}u_0\|_{L^a(0,\infty; L^{\varrho,q}(\R^N))} +C \|u\|^{\alpha+1}_{ L^{a}(0,\infty;L^{\varrho,q}(\R^N))}.
$$
By \eqref{similar2.9} we have
\begin{eqnarray*}
\|u\|_{ L^{ a}(0,\infty;L^{\varrho,q}(\R^N))}&\leq & \|e^{it\Delta}u_0\|_{L^a(0,\infty; L^{\varrho,q}(\R^N))} +C2^\alpha\varepsilon^\alpha\|u\|_{ L^{a}(0,\infty;L^{\varrho,q}(\R^N))}\\&\leq & \|e^{it\Delta}u_0\|_{L^a(0,\infty; L^{\varrho,q}(\R^N))} +{1\over 2} \|u\|_{ L^{a}(0,\infty;L^{\varrho,q}(\R^N))}.
\end{eqnarray*}
and \eqref{estimatinglobalerapid} follows.

Using \eqref{similar2.7}, we now write
\begin{eqnarray*}
\|u\|_{ L^{r}(0,\infty;W_2^{s,p}(\R^N))}&\leq & C\|u_0\|_{H^s(\R^N)}+ C \|u\|^{\alpha}_{ L^{a}(0,\infty;L^{\varrho,q}(\R^N))}\|u\|_{ L^{\varsigma}(0,\infty;W_2^{s,\varrho}(\R^N))}\\ &\leq & C\|u_0\|_{H^s(\R^N)}+ C2^\alpha\varepsilon^\alpha\|u\|_{ L^{\varsigma}(0,T;W_2^{s,\varrho}(\R^N))}\\ &\leq & C\|u_0\|_{H^s(\R^N)}+ {1\over 2}\|u\|_{ L^{\varsigma}(0,T;W_2^{s,\varrho}(\R^N))},
\end{eqnarray*}
and \eqref{estimatinglobalestrit} follows by \eqref{similar2.10}.

 This completes the proof of  Theorem \ref{global}.
\end{proof}

We now give examples  of initial data $u_0$ such that $\|e^{it\Delta}u_0\|_{L^a(\R; L^{\varrho,q}(\R^N))}<\infty.$
We have the following.
\begin{corollary}
\label{exempleinitialdata} Assume that $N\geq 4$. Let $K$ be a complex valued function satisfying $(K_1)$ in the  $L^2$-local theory case or $(K_1)-(K_2)$ in the  $H^1$-local theory case. Let $$\alpha_0(b)<\alpha<{4-2b\over N-2},$$
 $a,\; \varrho$ be given by \eqref{aP} and $2\leq q\leq \infty$. Let   $\varphi\in L^{\varrho',q}$ with $\varphi\in H^1$ or the Fourier transform  $\hat{\varphi}$ of $\varphi$ belongs to $L^{\varrho',q}$. Then $e^{it\Delta}\varphi\in L^a\left(\R; L^{\varrho,q}(\R^N)\right)$. Moreover the following hold.
 \begin{itemize}
 \item[(i)] If $\varphi\in L^{\varrho',q}\cap H^1$, then the $H^1$-maximal solution $u$ of \eqref{INLS} with initial value $u_0=\lambda \varphi$ for $|\lambda|$ sufficiently small is global.
 \item[(ii)] If $\varphi$, $\hat{\varphi}\in L^{\varrho',q}$ and $\alpha\leq {4-2b\over N}$ then the $L^2$-maximal solution $u$ of \eqref{INLS} with initial value $u_0=\lambda \varphi$ for $|\lambda|$ sufficiently small is global. Here we can take $N=1,\,2,\, 3$ with $0\leq b < \min(2,N)$.
  \item[(iii)] If $\varphi$ is as in (i) or (ii), then the maximal solution $u$ of
   \eqref{INLS} with initial value $\psi_s(x)=e^{{i s\Delta}}\varphi$ for $s$ sufficiently large is   positively global respectively as an $H^1$-solution and $L^2$-solution.
\end{itemize}
Moreover in all cases,  $u\in L^a\left(0,\infty; L^{\varrho,q}(\R^N)\right)$  and  scatters in  $H^\sigma$ as $t\to\infty$, for $u_0\in H^\sigma$, $\sigma=0,1$.
\end{corollary}
\begin{proof}
 Since $\varphi\in L^{\varrho',q}$, then by the dispersive estimate there exists a constant $C>0$ such that
\begin{equation}
\label{decroissanceSigma}
\|e^{it\Delta}\varphi\|_{L^{\varrho,q}(\R^N)}\leq C|t|^{-N({1\over 2}-{1\over \varrho})}\|\varphi\|_{L^{\rho',q}(\R^N)}, \; \forall\; t\not=0.
\end{equation}

 If $\varphi\in H^1(\R^N),$ then $e^{it\Delta}\varphi\in  L^\infty (\R,H^1(\R^N))\hookrightarrow L^\infty (\R,L^{p,q}(\R^N))$ for every $2\leq p\leq {2N\over N-2}.$ this gives that there exists $C>0$ such that
 \begin{equation}
\label{decroissanceSigma22}
\|e^{it\Delta}\varphi\|_{L^{\varrho,q}(\R^N)}\leq C\|\varphi\|_{H^1(\R)},\; \forall\; t\in \R.
\end{equation}
Using the fact that $Na({1\over 2}-{1\over \varrho})>1$ which is realized if and only if $\alpha>\alpha_0(b)$ and combining \eqref{decroissanceSigma22} and \eqref{decroissanceSigma} we obtain that there exists $C>0,$ such that  $$\|e^{it\Delta}\varphi\|_{L^a(\R; L^{\varrho,q})}\leq C\left(\|\varphi\|_{H^1}+\|\varphi\|_{L^{\rho',q}}\right).$$

 If $\hat{\varphi }\in L^{\varrho',q}$ then we have
\begin{eqnarray*}
\|e^{it\Delta}\varphi\|_{L^{\varrho,q}(\R^N)}&\leq &C\|\widehat{e^{it\Delta}\varphi}\|_{L^{\varrho',q}(\R^N)}\\  & \leq &C\|e^{-it|\cdot|^2}\hat{\varphi}\|_{L^{\varrho',q}(\R^N)}\\
&= &C\|\hat{\varphi}\|_{L^{\varrho',q}(\R^N)},\quad\forall t\in \mathbb{R}.
\end{eqnarray*}
Combining this estimate with \eqref{decroissanceSigma}, we get
$$\|e^{it\Delta}\varphi\|_{L^a(\R; L^{\varrho,q})}\leq C\left(\|\hat{\varphi}\|_{L^{\varrho',q}(\R^N)}+\|\varphi\|_{L^{\rho',q}}\right).$$
Which gives the desired result.

The proof of (i) and (ii) are then obvious. For (iii) we have
\begin{eqnarray*}
\|e^{it\Delta}\psi_s\|^a_{L^a(0,\infty; L^{\varrho,q})}&=&\int_0^{\infty}\|e^{i(\tau+s)\Delta}\varphi\|^a_{L^{\varrho,q}}d\tau\\&=&\int_s^{\infty}\|e^{i\tau\Delta}\varphi\|^a_{L^{\varrho,q}}d\tau
\to 0 \mbox{ as } s\to \infty.
\end{eqnarray*}
  The last statements follows by Theorem \ref{global}.
\end{proof}

We may also construct another example of positively global solutions which blows up for negative time.

\begin{remark}{\rm  If $\varphi\in \Sigma$ then $\varphi \in H^1\cap L^{\varrho ',2}$. In fact, since $x\varphi\in L^2$, so by the H\"older enequality we have
\begin{eqnarray*}
\|\varphi\|_{ L^{\frac{2N}{N+2},2}}\leq \||x|\varphi\|_{ L^{2}}\||x|^{-1}\|_{ L^{N,\infty}}.
\end{eqnarray*} Using that $\frac{2N}{N+2}\leq \varrho '\leq 2$, then the result follows by interpolation.
}

\end{remark}

\begin{corollary}[Positively global existence]
\label{occilating}
Assume that $N\geq 4$. Let $K$ be a complex valued function satisfying $(K_1)$ in the  $L^2$-local theory case or $(K_1)-(K_2)$ in the  $H^1$-local theory case. Let $\alpha_0(b)<\alpha<{4-2b\over N-2},$ where $\alpha_0(b)$ is defined by \eqref{alpha0}.  Let $\varrho,\; a$ be given by \eqref{aP}. Then the following hold.
\begin{enumerate}
\item[(i)] If $\varphi\in \Sigma$, then there exists $d_0<\infty$ such that for every $d\geq d_0,$ the $H^1-$maximal solution $u$ of \eqref{INLS} with initial value
$\varphi_d(x)=e^{{id|x|^2\over 4}}\varphi(x),$ is   positively global. Furthermore $u\in L^a\left(0,\infty; L^{\varrho,2}(\R^N)\right)$ and it scatters in $\Sigma$ as $t\to\infty$. In the particular case $K(x)=\mu |x|^{-b}$ with $\mu<0$,  $${4-2b\over N}\leq \alpha<{4-2b\over N-2}$$ and $E(\varphi)<0$ where $E$ is defined by \eqref{energy}, the maximal existence time of $u$  is infinite and the minimal existence time  satisfies $T_{\min}(\varphi_d)<\infty.$
\item[(ii)]  Let $N\geq 1$ and $\alpha_0 (b)<\alpha\leq {4-2b\over N}$. If $\varphi\in L^2$ such that $\hat{\varphi}\in L^{\varrho',q}$  then the $L^2$-maximal solution $u$ of \eqref{INLS} with initial value $\varphi_d(x)=e^{{id|x|^2\over 4}}\varphi(x),$  for $d$ sufficiently large is global. Furthermore $u\in L^a\left(0,\infty; L^{\varrho,2}(\R^N)\right)$ and it scatters in $L^2$ as $t\to\infty$.
\end{enumerate}
\end{corollary}
\begin{proof} (i) The proof is similar to that of \cite[Corollary 2.5, p. 83]{CW3}. We give it for completeness.

 Let  $(\varsigma,\varrho)$ be the admissible pair, where $\varsigma=4(\alpha+2)/(N\alpha+2b).$   Using  the fact that dilations have no effect on the second index  in Lorentz spaces, we have that
\begin{eqnarray*}
\|e^{it\Delta}\varphi_d\|^a_{L^a(0,\infty; L^{\varrho,2})}&=&\int_0^{1\over d}(1-d\tau)^{2(a-\varsigma)/\varsigma}\|e^{i\tau\Delta}\varphi\|^a_{L^{\varrho,2}}d\tau\\
&\leq & C\int_0^{1\over d}(1-d\tau)^{2(a-\varsigma)/\varsigma}\|e^{i\tau\Delta}\varphi\|^a_{H^1(\R^N)}d\tau\\
&\leq & C\|\varphi\|^a_{H^1(\R^N)}\int_0^{1\over d}(1-d\tau)^{-2(\varsigma-a)/\varsigma}d\tau<\infty,
\end{eqnarray*}
since $2(\varsigma-a)/\varsigma<1.$ Hence $\|e^{it\Delta}\varphi_d\|^a_{L^a(0,\infty; L^{\varrho,2})}\to 0$ as $d\to \infty.$ We have $\varphi_d\in H^1,$ since $\varphi\in \Sigma.$ So the first statement of (i) follows by Theorem \ref{global}.

 The proof of the second statement of (i) is similar to that of \cite[Remark 2.6, p. 83]{CW3}. For reader convenience we give it here. By changing $\varphi_d$ into $\bar{\varphi_d}$, it suffices to show that if $E(\varphi)<0$ then $T_{\max}(\varphi_{-d})<\infty$ for all $d>0.$ Here $T_{\max}(\varphi_{-d})$ is the positive maximal existence time for the solution $u$ of \eqref{INLS} with initial data $\varphi_{-d}(x)=e^{-{id|x|^2\over 4}}\varphi(x).$ Define the positive function $f$ by
$$f(t)=\||\cdot|u(t)\|^2_{L^2(\R^N)}.$$
By \cite[Proposition 4.1, p. 201]{Farah}, $f\in C^2([0,T_{\max}(\varphi_{-d})))$ and we have
$$f'(t)=4Im\int_{\R^N}\bar{u}x.\nabla udx:=4F(u(t)),$$
$$f''(t)=16E(\varphi_{-d})+4\mu{(N\alpha-4+2b)\over \alpha+2}\int_{\R^N}|x|^{-b}|u|^{\alpha+2}dx.$$ Integrating the last inequality twice in time,
we get
$$f(t)= f(0)+tf'(0)+8t^2E(\varphi_{-d})+4\mu{(N\alpha-4+2b)\over \alpha+2}\int_0^t\int_0^s\int_{\R^N}|x|^{-b}|u|^{\alpha+2}(\sigma,x)dx d\sigma ds.$$ Since $\mu<0,\; \alpha\geq (4-2b)/N,$ we write
\begin{equation}
\label{ineqf}
f(t)\leq  f(0)+tf'(0)+8t^2E(\varphi_{-d}), \; t\in [0,T_{\max}(\varphi_{-d})).
\end{equation}
Define the polynomial $P$ by
$$P(t)=f(0)+tf'(0)+8E(\varphi_{-d})t^2, t\geq 0.$$ By inequality \eqref{ineqf}, we have
\begin{equation}
\label{compfP}
0\leq f(t)\leq P(t), \; t\in [0,T_{\max}(\varphi_{-d})).
\end{equation}
Using \cite[formula (4.3), (4.4)]{CW2}, we write
 $$f'(0)=4F(\varphi_{-d})=4F(\varphi)-{2d}\||\cdot|\varphi\|^2_{L^2(\R^N)},\; E(\varphi_{-d})=E(\varphi)-{d\over 2}F(\varphi)+{d^2\over 8}\||\cdot|\varphi\|^2_{L^2(\R^N)}.$$ Then we get
$$P(t)=\||\cdot|\varphi\|^2_{L^2(\R^N)}+4t\left(F(\varphi)-{d\over 2}\||\cdot|\varphi\|^2_{L^2(\R^N)}\right)+8t^2\left(E(\varphi)+{d^2\over 8}\||\cdot|\varphi\|^2_{L^2(\R^N)}-{d\over 2}F(\varphi)\right).$$
Clearly $P(1/d)={8\over d^2}E(\varphi)<0.$ Hence by inequality \eqref{compfP} $T_{\max}(\varphi_{-d})\leq 1/d<\infty.$ that is $T_{\min}(\varphi_{d})<\infty.$ Then $u$ blows up in finite time.  This completes the proof of Corollary \ref{occilating}.

(ii)  If $\hat{\varphi}\in L^{\varrho',q}$ then we have $\|e^{it\Delta}\varphi\|_{L^{\varrho,q}(\R^N)}\leq C\|\hat{\varphi}\|_{L^{\varrho',q}(\R^N)}.$ Hence the result follows as in the proof of (i).

\end{proof}
\begin{remark}
$\,$ {\rm
\begin{itemize}
\item[1)] The results of Corollaries \ref{exempleinitialdata}, \ref{occilating} are known for $b=0$. See \cite{CW3}.
\item[2)] The result of  Corollary \ref{occilating}  is not covered by Theorems 1.5 and 1.6 of \cite{Farah}.
\item[3)] If we choose $\varphi\in \Sigma$ such that $\|\nabla \varphi\|^{\beta}_{L^2}\|\varphi\|^{1-\beta}_{L^2}>E(Q)^{\beta}\|Q\|_{L^2}^{1-\beta}$, where
$\beta=\frac{N}{2}-\frac{2-b}{\alpha}$ then  the initial data $e^{is\Delta}\varphi$ does not satisfy the conditions 1.11 and 1.12 of \cite{Farah} for $s$ sufficiently large.
\item[4)] The result of Corollary \ref{exempleinitialdata} (ii)-(iii) seems to be interesting for $Im(K)<0$. In fact, in this case the $L^2$-norm of the solution is  increasing and hence it may blow-up.
\end{itemize}}
\end{remark}

\section{Scattering criteria}
In this section we prove the  scattering criteria of Theorems \ref{globalstritchartzu} and \ref{globalstritchartzuw}.
Before giving the proofs of the theorems we establish the following.
\begin{lemma}
\label{importantestimates}
Assume that $N\geq 4,\; 0\leq b<2$ and   $$0<\alpha<{4-2b\over N-2}.$$
Let $K$ be a function satisfying $(K_1)$ and $(K_2)$ or $(K_1)$ and  $(K_3).$ Set $\tilde{K}=|x|^{-b}$ for the first case and $\tilde{K}=K$, for the second case. Let $\varrho$ be given by \eqref{aP} and  $u$ be such that $|\tilde{K}|^{{1\over \alpha+2}}u\in L^{\alpha+2,\infty}(\R^N).$ Then there exists a constant $C>0$ such that the following hold.
\begin{itemize}
\item[(i)] For $v\in L^{\varrho,2}(\R^N)$ we have $$\|K|u|^\alpha v\|_{L^{\varrho',2}(\R^N)} \leq C \||\tilde{K}|^{{1\over \alpha+2}}u\|^{\alpha}_{L^{\alpha+2,\infty}(\R^N)} \|v\|_{L^{\varrho,2}(\R^N)}.$$
    \item[(ii)] For $\nabla v\in L^{\varrho,2}(\R^N)$ we have $$\|\nabla K |u|^\alpha v\|_{L^{\varrho',2}(\R^N)}\leq C \||\tilde{K}|^{{1\over \alpha+2}}u\|^{\alpha}_{L^{\alpha+2,\infty}(\R^N)} \|\nabla v\|_{L^{\varrho,2}(\R^N)}.$$
\end{itemize}
In particular if  in addition $u\in W_2^{1,\varrho}(\R^N),$ then there exists a constant $C>0$ such that
\begin{equation}
\label{estimatecercle}
\|K|u|^\alpha u\|_{W_2^{1,\varrho'}(\R^N)} \leq C \||\tilde{K}|^{{1\over \alpha+2}}u\|^{\alpha}_{L^{\alpha+2,\infty}(\R^N)}\|u\|_{W_2^{1,\varrho}(\R^N)}.
\end{equation}
\end{lemma}
\begin{proof} Let us define the following numbers
$$ b_2={2 b\over \alpha+2},\; p_1={\alpha+2\over \alpha},\;  {1\over p_2}={1\over \varrho}-{1\over N}.$$

(i) Using the H\"older inequality, we have
\begin{eqnarray}
\|K|u|^\alpha v\|_{L^{\varrho',2}(\R^N)}  &\lesssim & \||\cdot|^{-{b_2}}\|_{L^{{N\over b_2},\infty}(\R^N)}\||\tilde{K}|^{{\frac{\alpha}{\alpha+2}}}|u|^\alpha\|_{L^{p_1,\infty}(\R^N)} \|v\|_{L^{\rho,2}(\R^N)}\nonumber\\
&\lesssim & \||\tilde{K}|^{{1\over \alpha+2}}u\|^{\alpha}_{L^{\alpha+2,\infty}(\R^N)} \|v\|_{L^{\varrho,2}(\R^N)}\label{o}.
\end{eqnarray}

(ii) Using the Sobolev embedding, we have
\begin{eqnarray}
\|\nabla K |u|^\alpha v\|_{L^{\varrho',2}(\R^N)} &\lesssim &\||\cdot|^{-b_2-1}|\tilde{K}|^{\frac{\alpha}{\alpha+2}}|u|^\alpha v\|_{L^{\varrho',2}(\R^N)} \nonumber\\ &\lesssim & \||\cdot|^{-{b_2-1}}\|_{L^{{N\over b_2+1},\infty}(\R^N)}\||\tilde{K}|^{\frac{\alpha}{\alpha+2}}|u|^\alpha\|_{L^{p_1,\infty}(\R^N)} \|v\|_{L^{p_2,2}(\R^N)}\nonumber\\ &\lesssim & \||\tilde{K}|^{\frac{1}{\alpha+2}}u\|^{\alpha}_{L^{\alpha+2,\infty}(\R^N)} \|\nabla v\|_{L^{\varrho,2}(\R^N)}\label{oo}.
\end{eqnarray}

For the last estimate, we write
$$\|K|u|^\alpha u\|_{W_2^{1,\varrho'}(\R^N)} \lesssim  \|K|u|^\alpha u\|_{L^{\varrho',2}(\R^N)} +\|\nabla K |u|^\alpha u\|_{L^{\varrho',2}(\R^N)}+\|K|u|^\alpha \nabla u\|_{L^{\varrho',2}(\R^N)}.$$
Then by replacing $v$ by $u$ or $\nabla u$ in the above estimates, we get the desired inequality. This completes the proof of the lemma.
\end{proof}

We now turn to the proof of Theorem \ref{globalstritchartzuw}.
\begin{proof}[Proof of Theorem \ref{globalstritchartzuw}]
The proof uses some argument of \cite[Proposition 2.3, p. 81 and  Proposition 2.7, p. 83]{CW3}.  We divide the proof of the lemma into five steps.

{\bf Step 1.}  In this step we show that $u\in L^r([0,+\infty),W^{1,p}_2(\R^N)),$ for any admissible pair $(r,p).$ By \cite[Theorem 1]{AT-1}, we know that $u\in L^{r}(0,T;W_2^{1,p}(\R^N)),$ for any admissible pair $(r,p),$ and every $T<\infty.$  Let $0<T<t.$ We write
$$ u(t)=e^{i(t-T)\Delta}u(T)-i\int_T^te^{i(t-s)\Delta}(K|u|^\alpha u(s))ds.$$
Let $\varsigma$ be given by \eqref{varsigma}. By the Strichartz estimates with the pair $(\varsigma,\varrho)$, we have
$$\|u\|_{L^{\varsigma}([T,t],W^{1,\varrho}_2(\R^N))}\lesssim \|u(T)\|_{H^1(\R^N)}+\| K|u|^\alpha u\|_{L^{\varsigma'}([T,t],W^{1,\varrho'}_2(\R^N))}.$$ First we notice that since $\alpha<(4-2b)/(N-2)$ then $\varrho<(2N)/(N-2).$ Since $u\in H^1(\R^N)$ then  by Sobolev injection we have $u\in L^{\varrho,2}(\R^N).$ By the  hypothesis $(K_1)$ and using the H\"older inequality, we have
 \begin{eqnarray*}
 \|\tilde{K}^{{1\over \alpha+2}}u\|_{L^{{\alpha+2},\infty}(\R^N)} &\lesssim & \||\cdot|^{-{b\over \alpha+2}}u\|_{L^{{\alpha+2},\infty}(\R^N)} \\ &\lesssim & \|u\|_{L^{\varrho,\infty}(\R^N)}\\ &\lesssim & \|u\|_{L^{\varrho,2}(\R^N)}\\ &\lesssim & \|u\|_{H^1(\R^N)}.
 \end{eqnarray*}
 Hence $\tilde{K}^{{1\over \alpha+2}}u\in L^{\alpha+2,\infty}(\R^N)$.
By Lemma \ref{importantestimates}, we have
\begin{eqnarray*}
\|K|u|^\alpha u\|_{W^{1,\varrho'}_2(\R^N)}&\leq &  C \|\tilde{K}^{{1\over \alpha+2}}u\|^\alpha_{L^{\alpha+2,\infty}(\R^N)}\|u\|_{W_2^{1,\varrho}(\R^N)}.
\end{eqnarray*}
Hence, by the H\"older inequality in time and for $a$ given by \eqref{aP}, we have
\begin{eqnarray*}
\|K |u|^\alpha u\|_{L^{\varsigma'}([T,t],W^{1,\varrho'}_2(\R^N))}&\leq &  C\|\|\tilde{K}^{{1\over \alpha+2}}u\|^\alpha_{L^{\alpha+2,\infty}(\R^N)}\|_{L^{a\over \alpha}([T,t])}\|u\|_{L^{\varsigma}([T,t],W_2^{1,\varrho}(\R^N))}\\ &\leq &  C\|\tilde{K}^{{1\over \alpha+2}}u\|^{\alpha}_{L^a([T,t],L^{\alpha+2,\infty}(\R^N))}\|u\|_{L^{\varsigma}([T,t],W_2^{1,\varrho}(\R^N))}.
\end{eqnarray*}
We may choose $T$ sufficiently large so that  $C\|\tilde{K}^{{1\over \alpha+2}}u\|^{\alpha}_{L^a([T,t],L^{\alpha+2,\infty}(\R^N))}\leq 1/2$ to get
$$\|u\|_{L^{\varsigma}([T,t],W^{1,\varrho}_2(\R^N))}\lesssim C(\|u(T)\|_{H^1(\R^N)})+{1\over 2}\|u\|_{L^{\varsigma}([T,t],W_2^{1,\varrho}(\R^N)},$$
hence
$\|u\|_{L^{\varsigma}([T,t],W^{1,\varrho}_2(\R^N))}\leq 2C(\|u(T)\|_{H^1(\R^N)}).$
Letting $t\to \infty,$ we get
$$\|u\|_{L^{\varsigma}([T,\infty),W^{1,\varrho}_2(\R^N))}\leq 2C(\|u(T)\|_{H^1(\R^N)}).$$
We know by the local existence \cite[Theorem 1]{AT-1} that $u\in L^{\varsigma}([0,T),W^{1,\varrho}_2(\R^N)),$ so $u\in L^{\varsigma}([0,\infty),W^{1,\varrho}_2(\R^N)).$ The result for any admissible pair follows by Strichartz's estimates.

{\bf Step 2.} We show here that $u$ scatters in $H^1(\R^N).$ Write
$$u(t)=e^{it\Delta}u_0-i\int_0^te^{i(t-s)\Delta}(K|u|^\alpha u(s))ds.$$
Set
$$z(t):=e^{-it\Delta}u(t)=u_0-i\int_0^te^{-is\Delta}(K|u|^\alpha u(s))ds.$$
For $0<t<\tau,$ we have
$$z(t)-z(\tau)=-i e^{-it\Delta}\int_\tau^te^{i(t-s)\Delta}(K|u|^\alpha u(s))ds.$$
Let $(\varsigma,\varrho)$ be the admissible pair given by \eqref{aP} and \eqref{varsigma}. By the Strichartz estimates,  we have
\begin{eqnarray*}
\|z(t)-z(\tau)\|_{L^\infty((t,\tau),H^1(\R^N))}&=&\|e^{it\Delta}[z(t)-z(\tau)]\|_{L^\infty((t,\tau),H^1(\R^N))}
\\ &=& \left\|\int_\tau^te^{i(t-s)\Delta}(K|u|^\alpha u(s))ds\right\|_{L^\infty((t,\tau),H^1(\R^N))}\\ &\leq &
\|K|u|^\alpha u\|_{L^{\varsigma'}((t,\tau),W^{1,\varrho'}_2)}.
\end{eqnarray*}

Using \eqref{estimatecercle} in Lemma \ref{importantestimates}, the above step, H\"older's inequality in time, we get

\begin{eqnarray*}
\|z(t)-z(\tau)\|_{L^\infty((t,\tau),H^1(\R^N))} &\leq &
\left\|\|\tilde{K}^{{1\over \alpha+2}}u\|^\alpha_{L^{\alpha+2,\infty}(\R^N)}\right\|_{L^{a\over \alpha}((t,\tau))}\|u\|_{L^{\varsigma}((t,\tau),W^{1,\varrho}_2)}\\ &\leq &
\|\tilde{K}^{{1\over \alpha+2}}u\|^\alpha_{L^{a}((t,\tau),L^{\alpha+2,\infty}(\R^N))}\|u\|_{L^{\varsigma}(\R,W^{1,\varrho}_2)} \to 0,
\end{eqnarray*}
as $t,\; \tau\to \infty.$  Here $a$ is given by \eqref{aP}.

Thus $\|z(t)-z(\tau)\|_{H^1(\R^N)}\to 0$ as $t,\; \tau\to \infty.$ Then there exists $\varphi^+\in H^1(\R^N),$ such that $e^{-it\Delta}u(t)\to \varphi^+$ as $t\to \infty,$ in $H^1(\R^N).$

{\bf Step 3.}  In this step we show that $w:=(x+2it\nabla)u\in L^r_{loc}(\R,L^{p,2}(\R^N))\cap C(\R,L^{2}(\R^N)),$ for any admissible pair $(r,p).$ We use similar argument of \cite{Tao-Visan-Zhang}. See also \cite{Dinh1}. Let $T>0$ and $I=[0,T].$ We write $I=\bigcup_{j\; {\mbox{finite}}}I_j,$ where $I_j=[t_j,t_{j+1}],$  $|t_{j+1}-t_j|=|I_j|<\varepsilon,$ with $\varepsilon>0$ to be fixed later. We write
$$w(t)=e^{i(t-t_j)\Delta}w(t_j)-i\int_{t_j}^te^{i(t-s)\Delta}\left((x+2is\nabla)(K|u|^\alpha u(s)\right)ds.$$
We recall that $|(x+2is\nabla)(K|u|^\alpha u(s))|=2|s|\left|\nabla\left(e^{-i{|x|^2\over 4s}}K|u|^\alpha u(s)\right)\right|=2|s||\nabla\left(K|v|^\alpha v(s)\right)|,$ where $v(s,x)=e^{-i{|x|^2\over 4s}}u(s,x)$ as defined in \eqref{definitionv}. Also $2|t||\nabla v|=|w|$ by \eqref{definitionviden1}.

Let the following particular admissible pair $(\gamma,\rho),$  given by
\begin{equation}
\label{partAdmiss1}
\gamma=\frac{4(\alpha+2)}{\alpha(N-2)+2b},\quad  \quad \rho=\frac{N(\alpha+2)}{N+\alpha -b}.
\end{equation}
Using Strichartz's estimates, we have
\begin{eqnarray*}
\|w\|_{L^r(I_j,L^{p,2}(\R^N))}&\lesssim & \|w(t_j)\|_{L^{2}(\R^N)}+\|(x+2it\nabla)(K|u|^\alpha u(t)\|_{L^{\gamma'}(I_j,L^{\rho',2}(\R^N))}\\
&\lesssim & \|w(t_j)\|_{L^{2}(\R^N)}+\||t|\nabla(K|v|^\alpha v(t)\|_{L^{\gamma'}(I_j,L^{\rho',2}(\R^N))}\\ &\lesssim & \|w(t_j)\|_{L^{2}(\R^N)}+\||x|^{-b}|u|^\alpha w(t)\|_{L^{\gamma'}(I_j,L^{\rho',2}(\R^N))}\\&&+\||t||x|^{-b-1}|u|^\alpha |v|\|_{L^{\gamma'}(I_j,L^{\rho',2}(\R^N))},
\end{eqnarray*}
where the constants here are independent of $j.$

Using similar calculations as in \cite[Lemma 3.1]{AT-1}, we have
$$\||x|^{-b}|u|^\alpha w\|_{L^{\rho',2}(\R^N)}\lesssim \|\nabla u\|^\alpha_{L^{\rho,2}(\R^N)}\|w\|_{L^{\rho,2}(\R^N)}$$
and
$$\||x|^{-b-1}|u|^\alpha v\|_{L^{\rho',2}(\R^N)}\lesssim \|\nabla u\|^\alpha_{L^{\rho,2}(\R^N)}\|\nabla v\|_{L^{\rho,2}(\R^N)},$$
Hence since $2|t||\nabla v|=|w|,$ we have
$$\||t||x|^{-b-1}|u|^\alpha |v|\|_{L^{\rho',2}(\R^N)}\lesssim \|\nabla u\|^\alpha_{L^{\rho,2}(\R^N)}\|w\|_{L^{\rho,2}(\R^N)}.$$
We then get
\begin{eqnarray*}
\|w\|_{L^r(I_j,L^{p,2}(\R^N))}&\lesssim & \|w(t_j)\|_{L^{2}(\R^N)}+\|\|\nabla u\|^\alpha_{L^{\rho,2}(\R^N)}\|w\|_{L^{\rho,2}(\R^N)}\|_{L^{\gamma'}(I_j)}\\ &\lesssim & \|w(t_j)\|_{L^{2}(\R^N)}+\varepsilon^\delta\|\nabla u\|^\alpha_{L^{\gamma}(I_j,L^{\rho,2}(\R^N))}\|w\|_{L^{\gamma}(I_j,L^{\rho,2}(\R^N))}\\ &\lesssim & \|w(t_j)\|_{L^{2}(\R^N)}+\varepsilon^\delta\|\nabla u\|^\alpha_{S(\R,L^2)}\|w\|_{S(I_j,L^2)},
\end{eqnarray*}
where $\delta=1-{b\over 2}-{(N-2)\alpha\over 4}>0,$ and
$$\|\cdot\|_{S(I,L^2)}=\sup_{\{(r,p) \mbox{admissible pairs}\}} \|\cdot\|_{L^r(I,L^{p,2})}.$$
We now choose $\varepsilon>0$ such that $\varepsilon^\delta\|\nabla u\|^\alpha_{S(\R,L^2)}\leq 1/2,$ then get
$$\|w\|_{S(I_j,L^2)}\lesssim \|w(t_j)\|_{L^{2}(\R^N)}.$$
In particular $\|w(t_1)\|_{L^2}\lesssim \|w(t_0)\|_{L^{2}(\R^N)}$ and so on, we obtain $\|w(t_j)\|_{L^2}\lesssim \|w(t_0)\|_{L^{2}(\R^N)}$ for any $j.$ Hence
$\|w\|_{S(I_j,L^2)}\lesssim \|w(0)\|_{L^{2}(\R^N)},$ for any $j.$ This gives that $\|w\|_{L^r(I,L^{p,2}(\R^N))}\lesssim \|w(0)\|_{L^{2}(\R^N)}$. The continuity of $w$ follows by similar calculations as above; this completes the proof.

{\bf Step 4.}   In this step we show that $w\in L^r(\R,L^{p,2})$ for any admissible pair $(r,p).$ Let $0<T<t.$ We write
\begin{eqnarray*}
w(t)&=&e^{it\Delta}(xu_0)-i\int_0^te^{i(t-s)\Delta}\left((x+2is\nabla)K|u|^\alpha u(s)\right)ds\\
&=&w(T)-i\mu \int_T^te^{i(t-s)\Delta}\left((x+2is\nabla)K|u|^\alpha u(s)\right)ds.
\end{eqnarray*}
We have,
\begin{eqnarray*}
|(x+2is\nabla)K|u|^\alpha u(s)|&\lesssim& |K(x+2is\nabla)|u|^\alpha u(s)|+|s\nabla |K||u|^\alpha u(s)|\\ &\lesssim&
|K||s||\nabla\left(e^{-i{|x|^2\over 2}}|u|^\alpha u(s)\right)|+|\nabla K||s||u|^{\alpha+1}\\ &\lesssim&
|K||s||v|^\alpha|\nabla v(s)|+|s||\nabla K||u|^{\alpha}|v|\\ &\lesssim&
|K||u(s)|^\alpha|w(s)|+|s||\nabla K||u(s)|^{\alpha}|v(s)|.
\end{eqnarray*}
By Lemma \ref{importantestimates} (i), we have
$$\|K|u(s)|^\alpha|w(s)|\|_{L^{\varrho',2}(\R^N)}\leq C \||\tilde{K}|^{{1\over \alpha+2}}u\|^\alpha_{L^{\alpha+2,\infty}(\R^N)}\|w\|_{L^{\varrho,2}(\R^N)},$$
and by Lemma \ref{importantestimates} (ii), we have
\begin{eqnarray*}
\|\nabla K |u(s)|^\alpha|s||v(s)|\|_{L^{\varrho',2}(\R^N)}&\leq & C \||\tilde{K}|^{{1\over \alpha+2}}u\|^\alpha_{L^{\alpha+2,\infty}(\R^N)}\||s|\nabla v\|_{L^{\varrho,2}(\R^N)}\\ &\leq & C \||\tilde{K}|^{{1\over \alpha+2}}u\|^\alpha_{L^{\alpha+2,\infty}(\R^N)}\|w\|_{L^{\varrho,2}(\R^N)}.
\end{eqnarray*}
Hence,
\begin{equation}\label{NAH2o}
\|(x+2is\nabla)K|u|^\alpha u(s)\|_{L^{\varrho',2}(\R^N)}\lesssim \||\tilde{K}|^{{1\over \alpha+2}}u\|^\alpha_{L^{\alpha+2,\infty}(\R^N)}\|w\|_{L^{\varrho,2}(\R^N)}.\end{equation}
Now, by the  Step 2, the Strichartz estimates and Theorem \ref{Decay Estimates}, we have
\begin{eqnarray*}
\|w\|_{L^{\varsigma}((T,t),L^{\varrho,2}(\R^N))}&\lesssim & \|w(T)\|_{L^{2}(\R^N))}+\|(x+2is\nabla)K|u|^\alpha u\|_{L^{\varsigma'}((T,t),L^{\varrho',2}(\R^N))} \\ &\lesssim &\|w(0)\|_{L^{2}(\R^N))}+\|\||\tilde{K}|^{{1\over \alpha+2}}u\|^\alpha_{L^{\alpha+2,\infty}(\R^N)}\|w\|_{L^{\varrho,2}(\R^N)}\|_{L^{\varsigma'}((T,t))}\\ &\lesssim & \|xu(0)\|_{L^{2}(\R^N))}+\|\||\tilde{K}|^{{1\over \alpha+2}}u\|^\alpha_{L^{\alpha+2,\infty}(\R^N)}\|_{L^{a/\alpha}((T,t))}\|w\|_{L^{\varsigma}((T,t),L^{\varrho,2}(\R^N)}\\ &\lesssim & \|xu(0)\|_{L^{2}(\R^N))}+\|\||\tilde{K}|^{{1\over \alpha+2}}u\|_{L^{\alpha+2,\infty}(\R^N)}\|^\alpha_{L^{a}((T,t))}\|w\|_{L^{\varsigma}((0,t),L^{\varrho,2}(\R^N)}.
\end{eqnarray*}
One concludes as in step 1 using the conclusion of Step 2. This completes the proof of Proposition \ref{globalstritchartzuw}.

{\bf Step 5.} Here we show that $u$ scatters in $\Sigma.$ We have
$$x(z(t)-z(\tau))=-i e^{-it\Delta}\int_\tau^te^{i(t-s)\Delta}(x+2is\nabla)(K|u|^\alpha u(s))ds.$$

Using Strichartz estimates in Lorentz spaces, we get
\begin{eqnarray*}
\|x(z(t)-z(\tau))\|_{L^\infty((t,\tau),L^2(\R^N))}&=&\|e^{it\Delta}[x(z(t)-z(\tau))]\|_{L^\infty((t,\tau),L^2(\R^N))}
\\ &=& \left\|\int_\tau^te^{i(t-s)\Delta}(x+2is\nabla)(K|u|^\alpha u(s))ds\right\|_{L^\infty((t,\tau),L^2(\R^N))}\\ &\leq &
C\|(x+2is\nabla)K|u|^\alpha u\|_{L^{\varsigma'}((t,\tau),L^{\varrho',2})}
\end{eqnarray*}

Using  Step 4, \eqref{NAH2o} and H\"older's inequality in time, we get

\begin{eqnarray*}
\|x(z(t)-z(\tau))\|_{L^\infty((t,\tau),L^2(\R^N))} &\leq &
\||\tilde{K}|^{{1\over \alpha+2}}u\|^\alpha_{L^{\alpha+2,\infty}(\R^N)}\|_{L^{a\over \alpha}((t,\tau))}\|w\|_{L^{\varsigma}((t,\tau),L^{\varrho,2})}\\ &\leq &
\|\tilde{K}|^{{1\over \alpha+2}}u\|_{L^{\alpha+2,\infty}(\R^N)}\|_{L^{a}((t,\tau))}\|w\|_{L^{\varsigma}(\R,L^{\varrho, 2})} \to 0,
\end{eqnarray*}
as $t,\; \tau\to \infty.$  Thus $\|x(z(t)-z(\tau))\|_{L^2(\R^N)}\to 0$ as $t,\; \tau\to \infty.$ Then  $\|x(e^{-it\Delta}u(t)-\varphi^+)\|_{L^2(\R^N)}\to 0$ as $t\to \infty.$

This completes the proof of the Theorem \ref{globalstritchartzuw}.
\end{proof}

\begin{proof}[Proof of Theorem \ref{globalstritchartzu}] It is similar to that  of Step 1 and Step 2 in the proof of Theorem \ref{globalstritchartzuw}. So it is omitted.

\end{proof}

\section{Decay estimates in a weighted $L^2$ space}
In this section we prove Theorem \ref{Decay Estimates}. For this, we first recall the following well known facts. See for example \cite{Dinh,Guzman}. Since $K\geq 0$,  the solution
of \eqref{INLS} with initial data $u_0\in H^1(\R^N)$ is global. Furthermore, the masse
\begin{equation}
\label{masse}
M(u(t))=\|u(t)\|_{L^2(\R^N)}^2
\end{equation}
and the energy
\begin{equation}
\label{energy1}
E(u)(t)={1\over 2}\|\nabla u(t)\|_{L^2(\R^N)}^2+{1\over \alpha+2}\left\|K|u(t)|^{\alpha+2}\right\|_{L^{1}(\R^N)}
\end{equation}
are conserved. The fact that $u\in C(\R, \Sigma)$ follows as in \cite{Farah,Cazenave,Dinh}.

We put
\begin{equation}
\label{energy2}
G(t)= {1\over \alpha+2}\left\| K|u(t)|^{\alpha+2}\right\|_{L^{1}(\R^N)}.
\end{equation}
We have the following identity  (\cite[Lemma 4.4, p. 15]{Dinh}),
\begin{eqnarray}
|(x+2it\nabla)u(t)\|_{L^2(\R^N)}^2+8t^2G(t)& =  &\|xu_0\|^2_{L^2(\R^N)}+4(4-N\alpha)\int_0^tsG(s)ds \nonumber \\& & + \frac{8}{\alpha+2}\int_0^t s\int_{\R^N} x.\nabla K(x)|u(s)|^{\alpha+2}dxds \label{identite1}.
\end{eqnarray}
Using hypothesis $(ii)$, we obtain
\begin{equation}
|(x+2it\nabla)u(t)\|_{L^2(\R^N)}^2+8t^2G(t)\leq   \|xu_0\|^2_{L^2(\R^N)}+4(4-2b-N\alpha)\int_0^tsG(s)ds.\label{identite12}
\end{equation}
Set
\begin{equation}
\label{definitionv}
v(t,x)=e^{-{i|x|^2\over 4t}}u(t,x).
\end{equation}
Then
\begin{equation}
\label{definitionviden1}
2ite^{{i|x|^2\over 4t}}\nabla v(t,x)=(x+2it\nabla)u(t,x),
\end{equation}
also
\begin{equation}
\label{definitionviden2}
\|v(t)\|_{L^{p,q}(\R^N)}=\|u(t)\|_{L^{p,q}(\R^N)},\;
2|t|\|\nabla v(t)\|_{L^{p,q}(\R^N)}=\|(x+2it\nabla)u(t)\|_{L^{p,q}(\R^N)}.\end{equation}
In particular,
$$\|(x+2it\nabla)u(t)\|^2_{L^2(\R^N)}=4t^2\|\nabla v(t)\|^2_{L^2(\R^N)}.$$
Then identity \eqref{energy1}  and inequality \eqref{identite12} lead to
\begin{eqnarray}
\nonumber
8t^2E(v)(t)&=& 4t^2\|\nabla v(t)\|^2_{L^2(\R^N)}+8t^2G(t)\\
\label{identite1new}
&\leq &\|xu_0\|_{L^2(\R^N)}^2+4(4-2b-N\alpha)\int_0^tsG(s)ds.
\end{eqnarray}

We now give the proof of Theorem \ref{Decay Estimates}.
\begin{proof}[Proof of Theorem \ref{Decay Estimates}]
$\,$\\
(i) Since $4-2b-N\alpha\leq 0,$ we deduce that
$$t^2\|\nabla v(t)\|^2_{L^2(\R^N)}\leq 2t^2E(v)(t)\leq {1\over 4}\|xu_0\|^2_{L^2(\R^N)}.$$
That is
\begin{equation}
\label{estimatinNablav}
\|\nabla v(t)\|_{L^2(\R^N)}\leq {1\over 2}\|xu_0\|_{L^2(\R^N)}|t|^{-1},\; t\not=0.
\end{equation}
Recall now that
$\|u_0\|_{L^2(\R^N)}=\|u(t)\|_{L^2(\R^N)}=\|v(t)\|_{L^2(\R^N)},$ for all $t\in \R.$ Then for $2<{\tilde{p}}<2N/(N-2),$ using the Gagiliardo-Nirenberg, we get
\begin{eqnarray*}
\|u(t)\|_{L^{{\tilde{p}},2}(\R^N)}&=&\|v(t)\|_{L^{{\tilde{p}},2}(\R^N)}\\ &\leq & \|\nabla v(t)\|_{L^2(\R^N)}^{N\left({1\over 2}-{1\over {\tilde{p}}}\right)}
\|v(t)\|_{L^2(\R^N)}^{1-N\left({1\over 2}-{1\over {\tilde{p}}}\right)}\\ &\leq & \left({1\over 2}
\|xu_0\|_{L^2(\R^N)}\right)^{N\left({1\over 2}-{1\over {\tilde{p}}}\right)}
\|u_0\|_{L^2(\R^N)}^{1-N\left({1\over 2}-{1\over {\tilde{p}}}\right)}
|t|^{-N\left({1\over 2}-{1\over {\tilde{p}}}\right)}.
\end{eqnarray*}
Hence \eqref{estimatelpq} follows. The estimate for ${\tilde{p}}=2N/(N-2)$  follows by the Sobolev embedding and \eqref{estimatinNablav}. The case $\tilde{p}=2$ is trivial.

(ii) We now suppose that $\alpha<(4-2b)/N,$ then $4-2b-N\alpha>0.$ From \eqref{identite1new} we write
$$8t^2E(v)(t)\leq \|xu_0\|_{L^2(\R^N)}^2+4(4-2b-N\alpha)\int_0^1sG(s)ds+4(4-2b-N\alpha)\int_1^tsG(s)ds.$$
Hence $$g(t):=t^2G(t)\leq C(u_0)+{(4-2b-N\alpha)\over 2}\int_1^t{1\over s}g(s)ds.$$
So, by Gronwall's inequality, we get
$g(t)\leq Ct^{-{N\alpha-4+2b\over 2}},$ for $t>1$
that is
\begin{equation}
\label{estimationG}
G(t)\leq Ct^{-{N\alpha+2b\over 2}}.
\end{equation}
Using  \eqref{identite1new} together with \eqref{estimationG} we get
\begin{equation}
\label{estimationv}
\|\nabla v(t)\|_{L_x^2(\R^N)}\leq Ct^{-{N\alpha+2b\over 4}}.
\end{equation}
The  estimates \eqref{estimationG} and \eqref{estimationv} are established in \cite{Dinh} for $K(x)=|x|^{-b}$. We use these two estimates in the rest of the proof.

We first consider the case $p\leq \alpha+2.$ Let $p_1,\; \theta$ be such that
$${1\over p_1}={1\over 2}+{b\over N(\alpha+2)},\;\theta\left({1\over p_1}-{1\over \alpha+2}\right)={1\over p_1}-{1\over p}.$$

Using interpolation and \eqref{estimationG}, we have
\begin{eqnarray*}
\left\|K^{{1\over \alpha+2}}u(t)\right\|_{L^{p,1}(\R^N)}&\leq & \left\|K^{{1\over \alpha+2}}u(t)\right\|_{L^{\alpha+2,\alpha+2}(\R^N)}^\theta
\left\||\cdot|^{-{b\over \alpha+2}}u(t)\right\|_{L^{p_1,2}(\R^N)}^{1-\theta}\\
 &\leq & \left\|K^{{1\over \alpha+2}}u(t)\right\|_{L^{\alpha+2,\alpha+2}(\R^N)}^\theta
\|u(t)\|_{L^{2}(\R^N)}^{1-\theta}\\
 &\leq & C\left\|K^{{1\over \alpha+2}}u(t)\right\|_{L^{\alpha+2,\alpha+2}(\R^N)}^\theta \\
 &\leq & C t^{-{N\alpha+2b\over 2(\alpha+2)}\theta},
\end{eqnarray*}
where we have used:
$$1<p_1\leq 2,\;0\leq \theta \leq 1,\; {1\over p}={\theta\over \alpha+2}+{1-\theta\over p_1}.$$
This proves \eqref{decy-blpq2} for $p\leq \alpha+2.$

We now consider the case $p>\alpha+2.$ Let $${1\over p_1}={b\over N(\alpha+2)}+{1\over 2}<1$$ and
$$\theta={2N(p-\alpha-2)\over p(4-2b-\alpha(N-2))}.$$
By the conditions on $p$ we have $\theta\in ]0,1[.$ It is clear that
$$1<p_1<N\; \mbox{ and }\;{1\over p}={\theta\over p_1}-{\theta\over N}+{1-\theta\over \alpha+2}.$$
By interpolation and the Sobolev embedding, we  get
\begin{eqnarray}
\label{ggnindecay}\left\|K^{{1\over \alpha+2}}u(t)\right\|_{L^{p,1}(\R^N)} &\leq & \nonumber C\left\|K^{{1\over \alpha+2}}u(t)\right\|_{L^{\alpha+2,\alpha+2}(\R^N)}^{1-\theta}\left\||\cdot|^{-{b\over \alpha+2}} v\right\|_{L^{{Np_1\over N-p_1},2}}^\theta\\&\leq & C\left\|K^{{1\over \alpha+2}}u(t)\right\|_{L^{\alpha+2,\alpha+2}(\R^N)}^{1-\theta}\left\|\nabla\left(|\cdot|^{-{b\over \alpha+2}} v\right)\right\|_{L^{p_1,2}}^\theta.
\end{eqnarray}
We have
$$\nabla\left(|\cdot|^{-{b\over \alpha+2}}v\right)=|\cdot|^{-{b\over \alpha+2}}\nabla v+C|x|^{-{b\over \alpha+2}-1}v.$$
On one hand, using H\"older's inequality in Lorentz spaces, we have
\begin{eqnarray*}
\||\cdot|^{-{b\over \alpha+2}}\nabla v\|_{L^{p_1,2}}&\leq & C\| |\cdot|^{-{b\over \alpha+2}}\|_{L^{{N(\alpha+2)\over b},\infty}(\R^N)}\|\nabla v(t)\|_{L^{2}(\R^N)}\\ &\leq & C\|\nabla v(t)\|_{L^{2}(\R^N)},\end{eqnarray*}
On the other hand,  H\"older's inequality  and the Sobolev inequality  in Lorentz spaces, give
\begin{eqnarray*}
\||\cdot|^{-{b\over \alpha+2}-1} v\|_{L^{p_1,2}}&\leq & C\| |\cdot|^{-{b\over \alpha+2}-1}\|_{L^{{N(\alpha+2)\over b+\alpha+2},\infty}(\R^N)}
\| v(t)\|_{L^{{2N\over N-2},2}(\R^N)} \\ &\leq & C \|\nabla v(t)\|_{L^2(\R^N)},
\end{eqnarray*}
where we used  $${1\over p_1}={b+\alpha+2\over N(\alpha+2)}+{N-2\over 2N}.$$
Then we get $$\left\|\nabla\left(|\cdot|^{-{b\over \alpha+2}} v\right)\right\|_{L^{p_1,2}}\leq C \|\nabla v(t)\|_{L^2(\R^N)}.$$
Using \eqref{ggnindecay}, \eqref{estimationG} and \eqref{estimationv}, we get
\begin{eqnarray*}
\left\|K^{{1\over \alpha+2}}u(t)\right\|_{L^{p,1}(\R^N)}&\leq & C\left\|K^{{1\over \alpha+2}}u(t)\right\|_{L^{\alpha+2,\alpha+2}(\R^N)}^{1-\theta}\|\nabla v(t)\|_{L^2(\R^N)}^\theta \\ &\leq & C
t^{-(1-\theta){N\alpha+2b\over 2(\alpha+2)}}t^{-\theta{N\alpha+2b\over 4}}\\ &= & C
t^{-{N\alpha+2b\over 4(\alpha+2)}(2+\alpha\theta)},
\end{eqnarray*}
giving the desired estimates. This completes the proof of Theorem \ref{Decay Estimates}.
\end{proof}

 \section{Scattering results}
In this section we give the proofs of the scattering results.  We begin by the case of oscillating initial data.

\begin{proof}[Proof of Corollary \ref{Scattering for oscillating solutions}]
By Theorem \ref{global}, we have $u\in L^a(0,\infty;L^{\varrho,q})$. We apply the H\"older inequality we get that $|x|^{-{b\over \alpha+2}}u\in L^{ a}(0,\infty;L^{{\alpha+2},\infty}(\R^N))$. Hence the result follows by Theorem \ref{globalstritchartzu} or Theorem \ref{globalstritchartzuw}.
 \end{proof}

 We now give the proof of Corollary \ref{EffectDecayK}.
\begin{proof}[Proof of Corollary \ref{EffectDecayK}]  Let $\alpha_0(\min (2,b_2))<\alpha <\frac{4-2b_1}{N-2s}$. The condition $\max (b_1,\frac{4-N\alpha^2-\alpha(N-2)}{2(\alpha+1)})\leq b\leq \min (b_2,\frac{4-(N-2s)\alpha}{2})$ implies that $\alpha_0(b)<\alpha <\frac{4-2b}{N-2s}$. If $b>0$, then by hypothese on $K$ we have $|K(x)|\lesssim |x|^{-b}$ and $|\nabla K(x)|\lesssim  |x|^{-b-1}$ if $s=1$. Hence by  Corollary \ref{Scattering for oscillating solutions} we have scattering for small initial data. If $b=0$ which implies $b_1=0$   we have $K, \nabla K\in L^\infty$, so by \cite{CW2} and Corollary \ref{Scattering for oscillating solutions}, the solution of \eqref{INLS}-\eqref{InitialINLS} with small initial data $u_0$, scatters. The last statement follows by the fact that $\alpha_0(2)=0$. This completes the proof.
\end{proof}

We now give the proof of Corollary \ref{scatering}.
\begin{proof}[Proof of Corollary \ref{scatering}] Let $u_0\in \Sigma.$ Applying Theorem \ref{Decay Estimates} with $p=\alpha+2,$ we have for some $A>0$:
$$\left\|\|K^{{1\over \alpha+2}}u (\cdot)\|_{L^{\alpha+2}(\R^N)}\right\|_{L^{a}(\{|s|>A\})}\lesssim \||s|^{-N({1\over 2}-{1\over \varrho})}\|_{L^{a}(\{|s|>A\})}.$$
Here $a$ is given by \eqref{aP}.
It is clear that the above right hand side is finite if $\alpha > \alpha_0(b)$.  Then, combined with the regularity of $u$ in time, we have $K^{{1\over \alpha+2}}u\in L^{ a}(\mathbb{R};L^{\alpha+2}(\R^N))\subset L^{ a}(\mathbb{R};L^{\alpha+2,\infty}(\R^N))$. The result follows by Theorem \ref{globalstritchartzuw}.
\end{proof}

 \bigskip
\section{Remarks on scattering in $H^1$}
In this section we give some results related to the scattering in $H^1$. We begin by the following criterion.
\begin{proposition}[Scattering criterion for $\alpha > \frac{4-2b}{N}$]
\label{globalstritchartzuw2}
Let $N\geq 4$, $0\leq b<2$, $$ {4-2b\over N} <\alpha< {4-2b\over N-2}$$ and $K$ be a real valued function   satisfying the condition $(K_1)-(K_2)$ or $(K_1)$ and $(K_3)$. Let $u\in C([0,\infty),H^1(\R^N))$ be a global solution of \eqref{INLS}. If  $\lim_{t\to \infty}\||\tilde{K}|^{{1\over \alpha+2}}u(t)\|_{L^{{\alpha+2,q}}(\R^N)}=0,$ (respectively  $\lim_{t\to \infty}\|u(t)\|_{L^{{\varrho,q}}(\R^N)}=0,$) for some $2\leq q\leq 2\alpha+2$ then $|\tilde{K}|^{{1\over \alpha+2}}u\in L^{ a}(0,\infty;L^{\alpha+2,q}(\R^N))$ (respectively $u\in L^{ a}(0,\infty;L^{\varrho,q}(\R^N))$ and the conclusion of Theorem \ref{globalstritchartzuw} holds for positive time.
 Similar statements holds for negative time.
\end{proposition}
\begin{proof} We first give the proof for the case $\lim_{t\to \infty}\||\tilde{K}|^{{1\over \alpha+2}}u(t)\|_{L^{{\alpha+2,q}}(\R^N)}=0$.  We have
\begin{eqnarray*}
\|u\|_{L^\varsigma (T,t,L^{\varrho, 2}(\R^N))}
&\lesssim & \|u(T)\|_{L^2(\R^N)} +\| K |u|^{\alpha+1}\|_{L^{\varsigma '} (T,t,L^{\varrho ', 2}(\R^N))}\\
&\lesssim & \|u(T)\|_{L^2(\R^N)} + \||\tilde{K}|^{{1\over \alpha+1}}u\|^{\alpha+1}_{L^{(\alpha+1)\varsigma '} (T,t,L^{(\alpha+1)\varrho ', 2(\alpha+1)}(\R^N))}\\
&\lesssim & \|u(T)\|_{L^2(\R^N)} + \||\tilde{K}|^{{1\over \alpha+2}}u\|^{\alpha+1}_{L^{(\alpha+1)\varsigma '} (T,t,L^{\alpha+2, 2(\alpha+1)}(\R^N))}\\
&\lesssim & \|u(0)\|_{L^2(\R^N)} + \sup_{s>T}\||\tilde{K}|^{{1\over \alpha+2}}u(s)\|^{\delta/\varsigma '}_{L^{\alpha+2, q}(\R^N))}\||x|^{-{b\over \alpha+2}}u\|^{\varsigma/\varsigma '}_{L^\varsigma (T,t,L^{\alpha+2,2}(\R^N))}\\ &\lesssim & \|u(0)\|_{L^2(\R^N)} + \sup_{s>T}\||\tilde{K}|^{{1\over \alpha+2}}u(s)\|^{\delta/\varsigma '}_{L^{\alpha+2, q}(\R^N))}\|u\|^{\varsigma/\varsigma '}_{L^\varsigma (T,t,L^{\varrho, 2}(\R^N))}.
\end{eqnarray*}
where $\delta=(\alpha+1)\varsigma '-\varsigma$. Since $\alpha>\frac{4-2b}{N}$ then  $\delta >0$. Using a classical  argument of continuity we can deduce that for $T$ sufficiently large we have
\begin{eqnarray*}
\|u\|_{L^\varsigma (T,t,L^{\varrho, 2}(\R^N))}& \lesssim & \|u(0)\|_{L^2(\R^N)}.
\end{eqnarray*}
But we know that $u\in L_{{\mbox{loc}}}^\varsigma (T,t,L^{\varrho, 2}(\R^N))$ so we obtain that  $u\in L^\varsigma (0,\infty,L^{\varrho, 2}(\R^N)).$ Then by the H\"older inequality in Lorentz spaces we have
 $$\||\tilde{K}|^{{1\over \alpha+2}}u\|_{L^\varsigma (0,\infty,L^{\alpha+2,2}(\R^N))} \lesssim  \|u\|_{L^\varsigma (0,\infty,L^{\varrho, 2}(\R^N))},$$
  and we get that $|\tilde{K}|^{{1\over \alpha+2}}u\in L^\varsigma (0,\infty,L^{\alpha+2, 2}(\R^N)).$ That is $|\tilde{K}|^{{1\over \alpha+2}}u\in L^\varsigma (0,\infty,L^{\alpha+2, q}(\R^N)).$ By the embedding of $H^1(\R^N)$ in $L^{\varrho,2}(\R^N),$ we have that $|\tilde{K}|^{{1\over \alpha+2}}u\in C(0,\infty,L^{\alpha+2, q}(\R^N))$. Now, since $\lim_{t\to \infty}\||\tilde{K}|^{{1\over \alpha+2}}u(t)\|_{L^{\alpha+2,q}(\R^N)}=0$ we obtain that $|\tilde{K}|^{{1\over \alpha+2}}u\in L^\infty(0,\infty,L^{\alpha+2, q}(\R^N)).$ Since $\alpha >\frac{4-2b}{N}$, then $\varsigma<a$ and by interpolation in time we conclude that $|\tilde{K}|^{{1\over \alpha+2}}u\in L^a (0,\infty,L^{\alpha+2, q}(\R^N))\subset L^a (0,\infty,L^{\alpha+2, \infty}(\R^N)).$ Hence Theorem \ref{globalstritchartzuw} applies.

Second we treat the case  $\lim_{t\to \infty}\|u(t)\|_{L^{{\varrho,q}}(\R^N)}=0$. We have, using the previous calculations,
\begin{eqnarray*}
\|u\|_{L^\varsigma (T,t,L^{\varrho, 2}(\R^N))}
&\lesssim & \|u(0)\|_{L^2(\R^N)} + \sup_{s>T}\||\tilde{K}|^{{1\over \alpha+2}}u(s)\|^{\delta/\varsigma '}_{L^{\alpha+2, q}(\R^N))}\|u\|^{\varsigma/\varsigma '}_{L^\varsigma (T,t,L^{\varrho, 2}(\R^N))}\\ &\lesssim &\|u(0)\|_{L^2(\R^N)} + \sup_{s>T}\|u(s)\|^{\delta/\varsigma '}_{L^{\varrho, q}(\R^N))}\|u\|^{\varsigma/\varsigma '}_{L^\varsigma (T,t,L^{\varrho, 2}(\R^N))}.
\end{eqnarray*}
By similar argument as above, we deduce that $u\in L^\varsigma (0,\infty,L^{\varrho, q}(\R^N))\cap L^\infty (0,\infty,L^{\varrho, q}(\R^N))$. The result follows by interpolation. This completes the proof of the proposition.
\end{proof}
\begin{remark}

\end{remark}{\rm
\begin{itemize}
\item[1)] We learn recently that in \cite{DK}  a different scattering criterion is given for the focusing case.
\item[2)] If the previous limit holds at $\pm\infty$   then the scattering occurs in both directions.
\end{itemize} }

 We have the following applications of the previous Proposition for $K(x)=\mu |x|^{-b}$, $\mu >0$.

\begin{corollary}[ $(H^1,H^1)-$ Scattering  for the defocusing case]
\label{refineddecayandscatteringH1}
Assume that   $N\geq 4,\; 0\leq b<2$, $K(x)=\mu |x|^{-b}$, $\mu>0,$ and  $${4-2b\over N}<\alpha<{4-2b\over N-2}.$$  Let $u_0\in H^1(\R^N)$ and $u\in C(\R,H^1(\R^N))$ be the global solution of \eqref{INLS}  with initial data $u_0.$ Then $u\in L^{ a}(\R;L^{\varrho,2}(\R^N))$ and hence  $u$ scatters in $H^1(\R^N).$
\end{corollary}

\begin{proof} By the H\"older inequality and interpolation, we have
$$\|u(t)\|_{L^{\rho,2}(\R^N)} \leq \|u(t)\|_{L^{\rho_1}(\R^N)}^\theta \|u\|_{L^{\rho_2}(\R^N)}^{1-\theta},$$
where
$${1\over \rho}={\theta\over \rho_1}+{1-\theta\over \rho_2},\; \theta\in (0,1),\; 2<\rho_1<\rho<\rho_2<{2N\over N-2}.$$
By \cite[Theorem 3, p. 415]{Dinh2}, we have $\lim_{t\to \pm\infty}\|u(t)\|_{L^{\rho_i}(\R^N)}=0,\; i=1,\; 2.$ Hence $\lim_{t\to \pm\infty}\|u(t)\|_{L^{\rho,2}(\R^N)}=0.$ The result follows then by Proposition \ref{globalstritchartzuw2}.
\end{proof}
\begin{remark}
{\rm The scattering result of the previous corollary is known \cite{Dinh2}. Our proof is different and gives the rapidly decay of the solution.}
\end{remark}


\begin{thebibliography}{99}
\bibitem{AT-1}{L. Aloui and S. Tayachi, {\em Local well-posedness for the inhomogeneous nonlinear Schr\"odinger equation}, Discrete $\&$ Continuous Dynamical Systems, doi: 10.3934/dcds.2021082.}
\bibitem{AT-3}{L. Aloui and S. Tayachi, {\em Local existence, global existence and scattering for the 3D inhomogeneous nonlinear Schr\"odinger equation}, preprint 2021.}

\bibitem{AIMMU}{K. Aoki, T. Inui, H. Miyazaki, H. Mizutani and K. Uriya, {\em Modified scattering for  inhomogeneous nonlinear Schr\"odinger equations with and without inverse-square potential}, Arxiv 2101.09423v2.}


\bibitem{DeBF}{A. De Bouard and R. Fukuizumi, {\em Stability of standing waves for nonlinear Schr\"odinger equations with inhomogeneous
nonlinearities}, Ann. Henri Poincar\'e, 6 (2005), 1157--1177.}

\bibitem{Cam}{ L. Campos, {\em Scattering of radial solutions to the inhomogeneous nonlinear Schr\"odinger equation}, Nonlinear Anal., 202
(2021), 112--118.}

\bibitem{Cam-Car}{ L. Campos and M. Cardoso {\em  Blow up and scattering criteria above the threshold for the focusing inhomogeneous
nonlinear Schr\"odinger equation}, preprint arxiv: 2001.11613.}

\bibitem{CFGM}{ M. Cardoso, L. G. Farah, C. M. Guzm\'an, and J. Murphy, {\em Scattering below the ground state for the intercritical
non-radial inhomogeneous NLS}, preprint arxiv:2007.06165.}

\bibitem{Cazenave} T. Cazenave, Semilinear Schr\"odinger Equations, Courant Lect. Notes Math., vol. 10, New York University, Courant Institute of Mathematical Sciences/Amer. Math. Soc., New York/Providence, RI, 2003.

\bibitem{CW2}{ T. Cazenave and F. B. Weissler, {\em The structure of solutions to the pseudo-conformally invariant nonlinear Schr\"odinger equation}, Proceedings of the Royal Society of Edinburgh, 117A (1991), 251--273.}

\bibitem{CW3}{ T. Cazenave and F. B. Weissler, {\em Rapidly decaying solutions of the nonlinear Schr\"odinger equation},
Comm. Math. Phys., 147 (1992), 75-100.}

\bibitem{Ch}{J. Chen, {\em On a class of nonlinear inhomogeneous Schr\"odinger equation}, J. Appl. Math. Comput. 32 (2010),
237--253.}

\bibitem{ChG}{ J. Chen and B. Guo, {\em Sharp global existence and blowing up results for inhomogeneous Schr\"odinger equations}, Discrete
Contin. Dyn. Syst. Ser. B 8 (2007), 357--367.}

\bibitem{CHL}{Y. Cho, S. Hong, and K. Lee, {\em On the global well-posedness of focusing energy-critical inhomogeneous NLS}, J. Evol.
Equ. 20 (2020), 1349-1380.}

\bibitem{CL}{Y. Cho  and K. Lee, {\em On the focusing energy-critical inhomogeneous {NLS}: weighted space approach}, Nonlinear Anal., 205 (2021), 112261, 21 pp.}

\bibitem{Dinh}{ V. D. Dinh, {\em Scattering theory in a weighted $L^2$ space for a class of the defocusing inhomegeneous nonlinear
Schr\"odinger equation}, preprint arXiv:1710.01392, 2017.}
\bibitem{Dinh1}{ V. D. Dinh, {\em Blowup of {$H^1$} solutions for a class of the focusing inhomogeneous nonlinear {S}chr\"{o}dinger equation}, Nonlinear Anal., 174 (2018), 169--188.}
\bibitem{Dinh2}{ V. D. Dinh, {\em Energy scattering for a class of the defocusing inhomogeneous nonlinear Schr\"odinger equation}, J. Evol. Equ., 19 (2019), 411-434.}

\bibitem{DK}{ V. D. Dinh and S. Keraani, {\em Long time dynamics of non-radial solutions to  inhomogeneous nonlinear Schr\"odinger equations}, preprint arXiv: 2105.04941.}



\bibitem{Farah}{ L. G. Farah, {\em Global well-posedness and blow-up on the energy space for the  inhomegeneous nonlinear
Schr\"odinger equation}, J. Evol. Equ. 16 (2016), 193--208.}
\bibitem{FarahGuzman}{ L. G. Farah and C. M. Guzm\'an, {\em Scattering for the radial 3{D} cubic focusing inhomogeneous nonlinear {S}chr\"{o}dinger equation}, J. Differential Equations, 262 (2017), 4175--4231.}

\bibitem{FarahGuzman2}{L. G. Farah and C. M. Guzm\'an, {\em Scattering for the radial focusing inhomogeneous NLS equation in higher dimensions},
Bull. Braz. Math. Soc. (N.S.) 51 (2020), 449--512.}

\bibitem{FO}{R. Fukuizumi and M. Ohta, {\em Instability of standing waves for nonlinear Schr\"odinger equations with inhomogeneous
nonlinearities}, J. Math. Kyoto Univ. 45 (2005), 145--158.}

\bibitem{Genoud}{F. Genoud, {\em Bifurcation and stability of travelling waves in self-focusing planar waveguides}, Adv.
Nonlinear Stud., 10 (2010), 357-400.}
\bibitem{GS}{F. Genoud, C. A. Stuart, {\em Schr\"odinger equations with a spatially decaying nonlinearity: existence and
stability of standing waves}, Discrete Contin. Dyn. Syst., 21(2008),  137--286.}
\bibitem{GinibreVelo}{Ginibre J. and Velo G., {\em On a class of nonlinear Schr\"dinger equations. I. The Cauchy problem, general case}, J. Funct. Anal., 32 (1979), 1-32.}
\bibitem{Guzman}{C. M. Guzm\'an, {\em On well posedness for the inhomogneous nonlinear Schr\"odinger equation}, Nonlinear Anal.
37 (2017), 249-286.}
\bibitem{HYZ}{H. Hajaiej, X. W. Yub and Z. C. Zhai, {\em Fractional Gagliardo-Nirenberg and Hardy inequalities under
Lorentz norms}, J. Math. Anal. Appl., 396 (2012),  569-577.}

\bibitem{KT}{M. Keel, T. Tao, {\em Endpoint Strichartz Estimates. American Journal of Mathematics}, 120(1998), 955-980.}

 \bibitem{KimLeeSeo}{J. Kim, Y; J. Lee and I, Seo, {\em On well-posedness for the inhomogeneous nonlinear Schr\"odinger equation in the critical case}, preprint, Arxiv: 1907.11871v1, July 2019.}

\bibitem{LeeSeo}{J. Lee and I, Seo, {\em The Cauchy problem for the energy-critical inhomogeneous nonlinear Schr\"odinger equation}, preprint, Arxiv: 1911.01112v2, November, 2019.}

 \bibitem{Lemarie}{P. G. Lemari\'e-Rieusset, {\em Recent developments in the {N}avier-{S}tokes problem}, Chapman \& Hall/CRC Research Notes in Mathematics., 431 (2002), Chapman \& Hall/CRC, Boca Raton, FL.}

\bibitem{LWW}{Y. Liu, X. Wang, and K. Wang, {\em Instability of standing waves of the Schr\"odinger equation with inhomogeneous nonlinearity}, Trans. Amer. Math. Soc. 358 (2006),  2105--2122.}

\bibitem{M}{ F. Merle, {\em Nonexistence of minimal blow-up solutions of equations $iu_t= -\Delta u- k(x)|u|^{4/N}u$ in $R^N$}, Ann. Inst. H. Poincar\'e Phys. Th\'eor. 64 (1996), 33--85.}

\bibitem{MMZ}{C. Miao, J. Murphy, and J. Zheng, {\em Scattering for the non-radial inhomogeneous NLS}, to appear in Mathematical
Research Letters, preprint arxiv: 1912.01318.}

\bibitem{RS}{ P. Rapha\"el and J. Szeftel, {\em Existence and uniqueness of minimal blow-up solutions to an inhomogeneous mass critical
NLS}, J. Amer. Math. Soc. 24 (2011), 471--546.}

\bibitem{St}{W. A. Strauss, {\em Nonlinear scattering theory at low energy}, J. Func. Anal., 41 (1981), 110--133.}

\bibitem{taggart}{R. J. Taggart, {\em Inhomogeneous {S}trichartz estimates}, Forum Math., 22 (2010), 825--853.}

\bibitem{Tao-Visan-Zhang}{T. Tao, M. Visan and X. Zhang, {\em The nonlinear Schr\"odinger equation with combined power nonlinearities}, Comm. Partal Differential Equations, 22(2007), 1281-1343.}

  \bibitem{Tsutsumi}{Y. Tsutsumi, {\em Scattering problem for nonlinear {S}chr\"{o}dinger equations}, Ann. Inst. H. Poincar\'{e} Phys. Th\'{e}or., 43 (1985), 321--347.}

\bibitem{Vilela}{M.C. Vilela, {\em Inhomogeneous Strichartz estimates for the Schr\"odinger equation}, Trans. Amer. Math. Soc. 359 (2007), 2123--2136.}

\bibitem{Z}{ S. Zhu, {\em Blow-up solutions for the inhomogeneous Schr\"odinger equation with $L^2$ supercritical nonlinearity}, J. Math.
Anal. Appl. 409 (2014), 760--776.}
\end{thebibliography}
\end{document}